\documentclass[a4paper,11pt]{amsart}
\usepackage{graphicx}
\usepackage{amsmath}
\usepackage{amsfonts}
\usepackage{amssymb}
\usepackage{amsthm}

\begin{document}

\title{Decay rate estimations for linear quadratic optimal regulators}

\author{Daniel Est\'evez}
\address{Departamento de Matem\'{a}ticas,
Universidad Autonoma de Madrid, Cantoblanco 28049 (Madrid) Spain}
\email{daniel@destevez.net}

\author{Dmitry V. Yakubovich}
\address{Departamento de Matem\'{a}ticas,
Universidad Autonoma de Madrid, Cantoblanco 28049 (Madrid) Spain
\enspace and
\newline
\phantom{r} Instituto de Ciencias
Matem\'{a}ticas (CSIC - UAM - UC3M - UCM)}
\email{dmitry.yakubovich@uam.es}

\date\today


\subjclass[2000]{Primary 93D05; Secondary 15A24}

\keywords{linear quadratic regulator; eigenvalue bounds; continuous algebraic
Riccati equation; exponential decay}

\maketitle

\begin{abstract}
Let $u(t)=-Fx(t)$ be the optimal control of the open-loop system
$x'(t)=Ax(t)+Bu(t)$ in a linear quadratic optimization problem.
By using different complex variable arguments, we give several lower and upper
estimates of the exponential decay rate
of the closed-loop system
$x'(t)=(A-BF)x(t)$.
Main attention is given to the case of a skew-Hermitian matrix $A$.
Given an operator $A$, for a class of cases, we find a matrix $B$ that provides an almost optimal decay rate.

We show how our results can be applied to the
problem of optimizing the decay rate for a large
finite collection of control systems $(A, B_j)$, $j=1, \dots, N$,
and illustrate this on an example of a concrete mechanical system.
At the end of the article, we pose several questions
concerning the decay rates in the context
of linear quadratic optimization and in a more general context
of the pole placement problem.
\end{abstract}

\

%
%
%
%
%
%
%
%
%
%

\

\noindent {\sc Highlights:}

\

\begin{itemize}

\item We give several lower and upper estimates of the decay rate for the
closed-loop system, arising from the linear quadratic optimal regulator problem
for a system
$(A,B)$, where $A$ is skew-Hermitian.

\item   For a class of cases, we find the control matrix $B$ that provides an
almost optimal decay rate.

\item Numerical examples of tightness of our estimates are given.

\end{itemize}

\renewcommand{\Re}{\operatorname{Re}}
\renewcommand{\Im}{\operatorname{Im}}

\renewcommand{\kappa}{\varkappa}
\renewcommand{\phi}{\varphi}

\theoremstyle{plain}
\newtheorem{theor}{Theorem}
\newtheorem{lemma}[theor]{Lemma}
\newtheorem{corol}[theor]{Corollary}

\newtheorem{example}[theor]{Example}

\theoremstyle{remark}
\newtheorem*{remark}{Remark}
\newtheorem*{remarks}{Remarks}

\theoremstyle{definition}
\newtheorem*{definition}{Definition}

\newcommand\defin {\overset {\text {\rm def} }{=}}
\newcommand\al{\alpha}
\newcommand\be{\beta}

\newcommand\dee{\delta}

\newcommand\De{\Delta}
\newcommand\de{\delta}
\newcommand\ga{\gamma}
\newcommand\Ga{\Gamma}
\newcommand\la{\lambda}
\newcommand\La{\Lambda}
\newcommand\om{\omega}
\newcommand\si{\sigma}
\newcommand\tht{\theta}
\newcommand\Ran{\operatorname{Ran}}
\newcommand\BC{\mathbb{C}}
\newcommand\BR{\mathbb{R}}
\newcommand\cl{\text{close}}
\newcommand\far{\text{far}}
\newcommand\ellest{\ell_{est}}
\newcommand\duc{d_{uc}}
\newcommand\lin{\operatorname{Lin}}

\newcommand\gafrob{\ga_{\mathrm{Frob}}}
\newcommand\gadecay{\gamma_{\mathrm{decay}}}
\newcommand\siclloop{\si_\text{cl.loop}}
\newcommand\almin{ \alpha_{\mathrm{min}} }
\newcommand\almax{ \alpha_{\mathrm{max}} }
\newcommand\SSS{S}


\section{Introduction}

It is well-known that in many practical problems, an engineer has to optimize,
in one or another
sense, several performance parameters of a control system. The Linear Quadratic Optimal Regulator
(LQR) problem searches a
stabilizing feedback which optimizes some associated quadratic cost
functional. Another important characteristic
of stabilization is the exponential decay rate of the resulting closed-loop
system. The main question
we address in this article is to study in which situations the LQR provides
good decay rates of the closed-loop system.

Recall that the standard Linear Quadratic Optimal Regulator problem
concerns the dynamic system of the form
\begin{equation}
\label{eq:system}
 x'(t) = Ax(t) + Bu(t), \quad x(0) = x_0.
\end{equation}
The problem is to  minimize the cost functional
\begin{equation}
\label{eq:j-diag}
 J^u(x_0) = \int_0^\infty
x(t)^*Qx(t)+
u(t)^*Ru(t)
 \,dt.
\end{equation}
Here $x(t)\in \BC^n$ is the state of the system and $u\in L^2_{\text{loc}}\big([0,+\infty),\BC^m\big)$ is a control function.
Matrices $A$, $B$, $R$, $Q$ are complex and have suitable sizes.
We assume that $R$ and $Q$ are positive definite.
We are specially interested in the case when the dimension $m$ of the control $u(t)$ is less than
$n$, the dimension of the state $x(t)$.

As is well-known (see \cite{Lanc-Rodm}, \cite{Willems}), the solution to the LQR
problem
is unique and the function $u(t)$, for which
the minimum of the cost functional is
attained is given by the feedback function $u(t) = -Fx(t) = -Fe^{(A-BF)t}x_0$,
where $F = R^{-1}B^*X$ is the feedback matrix and $X$ is any
nonnegative
solution of the continuous Algebraic Riccati Equation
\begin{equation}
\label{eq:original-care}
 XBR^{-1}B^*X - XA - A^*X - Q  = 0.
\end{equation}
This solution $X$ is unique and positive definite, and the minimum
cost functional is given by $\hat{J}(x_0) = x_0^* X x_0$. It is also notable
that the feedback matrix $F$ does not depend on $x_0$. The closed-loop system is
\begin{equation*}
\label{eq:orig-closedloop}
 x'(t) = A_{\text{cl.loop}}\,x(t),
\end{equation*}
where $A_{\text{cl.loop}} \defin A - BF$  is stable, that is, its
spectrum $\si_{\text{cl.loop}}$ lies in the open left half-plane
$\BC_-$. We denote by $\|\cdot\|$ the euclidean norm of vectors in
$\BC^k$ and the induced norm of matrices.

The linear quadratic problem is one of the most widespread methods for
stabilizing systems.
In this work, we give various estimates of the quality of this stabilization in
terms of the geometry of the spectrum of the open-loop system matrix $A$
and the characteristics of $B$.
We remark that the pole placement problem is known to be very ill conditioned
for control systems of large size and that the linear quadratic stabilization
is one of the methods for overcoming this difficulty. We refer to
\cite[Section 4]{HeLaubMehrm95}, \cite{MehrmXu2}, \cite{Chu98} and references therein
for theoretical results and for a discussion of different aspects of
the pole placement approach and its comparison with the linear quadratic approach
to stabilization.

The exponential decay rate of the closed-loop system is given by
\begin{equation}
\label{eq:gamma1}
 \gamma_{\mathrm{decay}}(A, B) = \min \big\{ |\Re \nu|: \enspace \nu\in
\siclloop\big\}.
\end{equation}
It is well-known that
\[
 \gadecay
 = \sup \,\{ \varepsilon>0 : \quad \forall x_0 \; \exists K=K(\varepsilon,x_0) :
 \enspace \|x(t)\| \leq Ke^{-\varepsilon t}, \forall t \geq 0\}.
\]
Hence $\gadecay$ can be seen as a characteristic of the
quality of the LQ control for large times $t$.
The LQ regulator can be considered to be good in this sense if
$\gamma_{\mathrm{decay}}$ is big.

\

The main results of this article concern upper and lower estimates of $\gadecay$.
This is done under the assumption that the matrix $A$ is skew-Hermitian:
$A^*=-A$ (that is, $iA$ is Hermitian). This assumption just means that under the absence of control
($u(t)\equiv 0$), the energy $\|x(t)\|^2$ is conserved.
Notice that if an open-loop linear system models a mechanical (or electrical)
system where the energy is conserved,
then we are in this situation.

We also will assume that
$$
Q=I, \quad R=I.
$$
The assumption about $Q$ is rather natural in view of the above remark on the
conservation of energy.
The case of $Q=|p(A)|^2$, where $p$ is a polynomial,
reduces easily to our setting. A general matrix weight $R>0$ is converted to
the the weight $R=I$ by making a linear substitution $\tilde u(t)=R^{1/2} u(t)$ in \eqref{eq:system}.

As we show, the upper and lower estimates of $\gadecay$ we give
permit one to compare the performance of the LQ optimal regulators of control
systems $(A, B_j)$, in which $A$ is fixed and there are several possibilities
for the matrix $B$.

We are not aware of any previous work estimating $\gadecay$ for LQ
optimal regulators. Other measures of the quality of control have
been studied already. Among the most
popular of them are the eigenvalues of $X$, $\|X\|$, $\operatorname{trace} X$
and $\det X$. Since $\hat{J}(x_0) = x_0^* X x_0$, these measures are
tightly related to the cost of the stabilized system.

Indeed, $\|X\|$ has the sense of the worst case performance of the cost
functional, for $x_0$ of fixed norm:
\[
\|X\| =\max_{\|x_0\| = 1} \hat{J}(x_0).
\]
Similarly, $n^{-1}\operatorname{trace} X$ is the average value of $\hat{J}(x_0)$
when $x_0$ ranges over the unit sphere. The larger is any of these measures of quality
of the control, the worse is the LQ stabilization.

Estimates for all these measures are well know. See for instance the reviews by Mori
and Derese \cite{MoriDerese}, and Kwon, Moon and Ahn \cite{KwonMoon}, the
papers \cite{Komaroff1989}, \cite{Mori85}, \cite{YasudaHir79} and
recent papers \cite{DaviesShietl07}, \cite{DaviesShietl08}, \cite{Lee03}, \cite{Lee06},
\cite{LiuZhangLiu}.

We observe the following easy relationship:
\begin{equation}
\label{eq:x-geq-gdec}
\gadecay \geq \frac{1}{2\, \|X\| }\, .
\end{equation}
This inequality is true because for any $\nu \in \siclloop$,
if $(A-BF)x_0 = \nu x_0$ and $\|x_0\|=1$, then
\[
\|X\|\ge
\langle Xx_0,x_0 \rangle \geq
 \int_0^\infty \|x(t)\|^2 dt = \frac{1}{2\,|\Re\nu|}\, .
\]
So any upper estimate of $\|X\|$ implies a lower estimate of $\gadecay$.
Several works give upper bounds for $\|X\|$, however, these bounds
are given under assumptions that either
 $A + A^* < 0$ or that $BB^*$ is
invertible. All our results deal with the case when $A+A^* =0$ and $BB^*$
can be singular.

Notice that
\eqref{eq:x-geq-gdec} shows that whenever the stabilization is bad in terms of the parameter $\gadecay$,
$\|X\|$ also is large.

We put
$$
\si(A)=\{i\la_1, \dots, i\la_n\}
$$
(where $\la_j\in \BR$) and assume throughout the whole article that
\begin{equation}
\label{eq:la-order}
\la_1\le \la_2 \le\dots\le  \la_n.
\end{equation}

Our estimates depend on the following numbers. The characteristic
\begin{equation}
 \label{eq:delta}
\delta(A) = \min_{j,k; j \neq k} |\lambda_j - \lambda_k|
\end{equation}
gives the minimal separation of eigenvalues. We
will write just $\delta$ when the dependence on $A$ is clear enough.
For a fixed index $k$, we put
\begin{equation}
\delta_k = \min_{j; j \neq k} |\lambda_j - \lambda_k|,
\end{equation}
which denotes the separation of the eigenvalue $i\la_k$ of $A$ from the rest. The
number
\begin{equation}
\label{eq:Delta}
\Delta=\Delta(A) = \max_{j,k} |\lambda_j - \lambda_k|=\la_n-\la_1
\end{equation}
will also be used.

The skew-Hermitian matrix $A$ can be diagonalized:
\begin{equation}
\label{eq:v-j}
Av_j=i\la_j v_j,
\end{equation}
where $\{v_j\}$ $(1\le j\le n)$ is an orthonormal basis of $\BC^n$. Put
\begin{equation}
\label{eq:def-bj}
b_j=B^* v_j.
\end{equation}

One of our main results can be stated as follows.

\begin{theor}
\label{thm:main}
Put
\begin{equation}
\label{def-ellest}
\ellest=
\min_{1\leq k \leq n}
\frac{\|b_k\|}{\sqrt{2}\,(1 + 2\frac{\|B\|^2}{\delta_k^2})}.
\end{equation}
Then the following statements hold.
\begin{enumerate}
\item The eigenvalues $\nu_j$ of the closed-loop system
lie in  the box $[-\|B\|,-\ellest)\times[\lambda_1,\lambda_n]$.

\item If moreover, $m\le n$ and the smallest singular value $\sigma_m$ of $B$
satisfies $\sigma_m > 2\sqrt{2} \Delta$, then  exactly
$m$ eigenvalues lie in box
$[-\|B\|,-\frac{\sqrt{6}}{4}\sigma_m]\times[\lambda_1,\lambda_n]$,
and the other $n-m$ eigenvalues lie in the box
$(-\sqrt{3}\Delta,-\ellest)\times[\lambda_1,\lambda_n]$.

\item In the case $m=1$, the bound $\ellest$ in the above assertions
can be improved by substituting it by a larger number
\begin{equation}
\label{eq:circles}
\ellest^1=
\min_{1\leq k \leq n}
\frac{\|b_k\|}{\sqrt{2}\cdot\sqrt{1 + 2\frac{\|B\|^2}{\delta_k^2}}}\; .
\end{equation}
\end{enumerate}
\end{theor}

In particular, it follows from this theorem that
\begin{equation}
\label{eq:gadec}
\gadecay > \ellest  \qquad (\gadecay > \ellest^1 \enspace \text{for} \enspace
m=1).
\end{equation}

If $A$ has multiple eigenvalues, we put $\ellest=\ellest^1=0$. It follows from the proof of this theorem
that all its statements remain true in this case.

It also follows from Theorem \ref{thm:main} that for $m=1$,
$\gadecay\le 2\sqrt{2}\, \De$, \textit{independently of the choice of the  $n\times 1$ matrix $B$}.
We will comment more on this phenomenon at the end of Section~\ref{mainresult}
and in Section~\ref{sec:open-que}, Question~1.

Theorem~\ref{thm:main-detailed} below gives a more detailed information about
the location of the spectrum of the closed-loop system.

Notice that the appearance of the norms of vectors
$b_k=B^*v_k$ in this estimate is very natural. In fact, the quantity
$$
d_0(A,B)=\min_k \|b_k\|
$$
can be taken for a kind of  \textit{measure of controllability} of the system $\dot x=Ax+Bu$.
In the case when all eigenvalues of $A$ are distinct, the system
is controllable if and only if $\min_k \|b_k\|>0$.
At the end of the Introduction, we will comment on the relation
between $d_0(A,B)$, the distance to uncontrollability $\duc(A,B)$, introduced
by Eising, and $\gadecay(A,B)$.

If $A$ is not normal, then one should use eigenvectors of $A^*$ instead of
eigenvectors of $A$ in the definition of the measure of controllability $d_0(A,B)$.

We remark that if $m = n$ and for some fixed $\beta > 0$
one can freely choose $B$ with $\|B\|=\beta$, then an optimal control
with the best possible
 $\gadecay$ can be given
easily. If $BB^* = \beta^2I$ (for example take $B = \beta I$), then the
solution to the associated continuous Algebraic Riccati Equation is $X = \beta^{-1}I$.
Hence, the closed-loop system matrix is $A_{\text{cl.loop}} = A -
\beta I$, and one can readily compute its eigenvalues. It follows that
in this case, in the bound $\gadecay \leq \|B\|$, which follows from
Theorem~\ref{thm:main}, the equality is attained.

We also observe that the case $m > n$ can be reduced to $m \le n$.
In fact, the optimal feedback $u(t)=-B^*Xx(t)$ ranges over the space $\Ran B^*$. Therefore the
linear quadratic problem for the pair $(A,B)$ reduces to the
same problem for the pair $\big(A, B |\Ran B^*\big)$; notice that $\dim \Ran B^*\le n$.
After this reduction, in place of $B$, we get the operator
$B |\Ran B^*$, which has trivial kernel.

For this reason, we will assume throughout the paper that
$$
m\le n \quad \text{and} \quad \ker B=0.
$$

Let us briefly overview the contents of the article by sections.
Section~\ref{mainresult} is devoted to the proof of
Theorem~\ref{thm:main-detailed}, which implies Theorem~\ref{thm:main} above.

In Section~\ref{sec:enough-sep}, we show that if the minimal separation
$\delta(A)$, defined in \eqref{eq:delta}, is rather big
in comparison with $\|B\|$, then the closed-loop eigenvalues of the system
can be located with good precision, which gives nice two-sided estimates of $\gadecay$.
In particular, Corollary~\ref{corol-phiks} shows that if $\|B\|/\de(A)$ is
rather small, then $\gadecay$ is comparable with $d_0(A,B)=\min\|b_k\|$.
In many problems of the design of optimal controllers,
the matrix $B$ can be changed, up so some extent.
In this section, for a given $A$, we find a ``suboptimal''
matrix $B$ among all
matrices with a fixed norm, which is supposed to be small.
(See Theorem~\ref{limit-case} and Corollary~\ref{limit-corol}.)

For $1 \leq k \leq n - 1$ we define
\begin{equation}
\label{Delta-m}
\Delta_k\defin\min_{1\le j \le n-k}|\lambda_{j+k}-\lambda_j|.
\end{equation}
Observe that with this notation, $\delta$ and $\Delta$ defined in equations
\eqref{eq:delta} and \eqref{eq:Delta} are $\delta = \Delta_1$ and $\Delta =
\Delta_{n-1}$.

In Section~\ref{sec:delta-m}, Theorem~\ref{decay-gener-m}, we give an estimate
of $\gadecay$ in terms of $\Delta_m$
(recall that $m$ is the dimension of $u(t)$). For $m>1$, this estimate may be much better
than the estimate of Theorem~\ref{thm:main}
if some of eigenvalues of $A$ are close to each other or coincide.

Section~\ref{briefaccount} contains a brief account of all our estimates of
$\gadecay$. In Section~\ref{sec:num-ex}, some numerical examples that illustrate
these estimates are given. In Subsection~\ref{4states2controls}, we give an
example in low dimension, which illustrates how our estimates compare
in different cases. In Subsection~\ref{mech}, we discuss the problem of
optimizing $\gadecay$ among a finite family $(A, B_j)$, with a fixed system
matrix $A$ and different possible choices for the control matrix $B$. We give an
algorithm which uses our estimates to reduce the number of computations
needed in the search. We illustrate this algorithm with a simple mechanical
system.

In Section \ref{sec:open-que}, we list some open questions, and in
Section~\ref{conclusions}, we list the conclusions of this article.

In what follows, we use the notation $\|Y\|_F$ for the Frobenius norm of a matrix $Y$. It is given by
\[
 \|Y\|^2_F = \operatorname{trace} (Y^*Y);
\]
this formula applies to rectangular matrices as well.

It is worth noticing that for general pairs of matrices $(A,B)$,
Eising introduced in \cite{Eising}
the so-called   ``distance to uncontrollability'', given by
\begin{equation}
\label{dist-uc}
\begin{aligned}
\duc=\duc(A,B) & =\inf
\Big\{
\left(
\|\de A\|^2_F+\|\de B\|^2_F\right)^{\frac{1}{2}}:
 \quad
\de  A\in \BC^{n\times n},\;  \de  B\in \BC^{n\times m},  \\
  & \qquad \qquad \quad (A+\de A, B+\de B) \enspace \text{uncontrollable} \Big\}.
\end{aligned}
\end{equation}
He proved that
$$
\duc(A,B) = \min_{\la\in \BC} \si_{\text{min}}(\big[A-\la I, B\big]),
$$
where $\si_{\text{min}}$ stands for the minimal
singular value.
Estimates for the quantity $\duc$ and methods for its computation
have been studied further in numerous works, see
\cite{BoleyLu86}, \cite{Demmel}, \cite{GahiLaub}, \cite{He1997}, \cite{Karow-Kressner},
\cite{SonThuan} and references therein.
Related characteristics were studied in the works \cite{VanLoan85},
\cite{HinrPritch1986b} and others.

It is not difficult to show that for any \textit{normal} matrix $A$ and for any
$B$
such that $m<n$, one has an estimate
\begin{equation}
\label{est-duc}
\duc(A,B)\le \min
\big(
d_0(A,B), r_{m+1}(\si(A))
\big),
\end{equation}
where $r_k\big(\si(A)\big)$ is the radius of the smallest disk containing at least $k$ points
of $\si(A)$. (If $iA$ is Hermitian, then $r_{m+1}(\si(A))=\frac 12 \De_{m}$.)

One gets from it a certain relationship between
$\gadecay$ and $\duc$ for $m=1$.
Indeed, if $m=1$, then by
\eqref{eq:gadec},
\begin{equation}
\label{est-d-delta}
\gadecay > \frac
{d_0(A,B)}{\sqrt{2}\cdot\sqrt{1+2\,\frac{\|B\|^2}{\de(A)^2}}}\,.
\end{equation}
By \eqref{est-duc}, $\duc\le \min \big(d_0(A,B), \de(A)/2\big)$, and we get
\begin{equation}
\label{gadecay-duc}
\gadecay > \frac {\duc^2}{\sqrt{\|B\|^2+2 \duc^2}}\ge \frac
{\duc^2}{\sqrt{3}\,\|B\|} \quad (m=1).
\end{equation}
We do not know whether an analogous estimate holds for $m>1$. One can observe
that
the characteristics $d_0(A,B)$ and $\de(A)$ of the system $(A,B)$ are in some
sense
independent. Therefore the estimate \eqref{est-d-delta}, which uses both
characteristics, gives in fact more information than \eqref{gadecay-duc}.

Numerical methods for solving matrix algebraic Riccati equations, in fact, have
been much investigated;
see \cite{Burns-etl}, \cite{PandKennLaub} 
and books by Sima \cite{Sima-book} and Datta \cite{Datta-book}.
We refer to \cite{HenchHeKuMehrm98}, \cite{SafonovAthans} and \cite{BorrelliKeviczky} for some
other interesting aspects of the linear quadratic problem.

\section{The main result on the location of closed-loop eigenvalues}
\label{mainresult}

The spectral theorem yields the decomposition
\begin{equation}
 \label{eq:A-decomposition}
A = \sum_{j=1}^n i\lambda_j v_jv_j^*,
\end{equation}
where the
eigenvalues $i \la_j$ of $A$ are assumed to satisfy \eqref{eq:la-order} and
the eigenvectors $v_j$  form an orthonormal
basis of $\mathbb{C}^n$ (see \eqref{eq:v-j}).
Moreover,
$B$ and $B^*$ decompose as
\begin{equation}
 \label{eq:decomp-b}
 B = \sum_{j=1}^n v_j b_j^*\qquad B^* = \sum_{j=1}^n b_jv_j^*,
\end{equation}
where the $b_j$'s have been defined in \eqref{eq:def-bj}.

For any index $k$, $1\le k\le n-1$, we consider the closed right triangle $T_k$
in $\mathbb{C}$ with vertices at the points
\[
i\la_k, \quad
i\la_{k+1}, \quad
- \frac {\la_{k+1}-\la_k} 2+ i\, \frac {\la_{k+1}+\la_k} 2\, .
\]
All these triangles lie in the half-plane $\Re z\le 0$
(see Figure \ref{zeros-area-small}). For any $k$, $1\le k\le n$, we put
\begin{equation}
 \label{eq:circles-m}
 \rho_k = \frac{\|b_k\|}{1 + 2\frac{\|B\|^2}{\delta_k^2}}\; ,
 \qquad
 \rho^1_k = \frac{\|b_k\|}{\sqrt{1 + 2\frac{\|B\|^2}{\delta_k^2}}}\; .
\end{equation}

\begin{figure}
\begin{center}
\includegraphics[width=220pt]{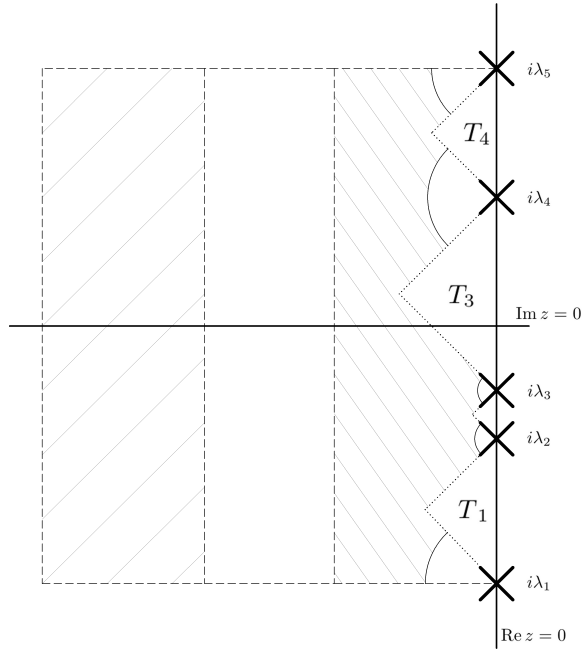}
\caption{Location of the eigenvalues of $A_{\text{cl.loop}}$.}
\label{zeros-area-small}
\end{center}
\end{figure}

Our next goal is to prove the following result.
\begin{theor}
\label{thm:main-detailed}
\begin{enumerate}
\item
\label{item}
The eigenvalues $\nu_j$ of the closed-loop system lie in  the
box
$[-\|B\|,0)\times[\lambda_1,\lambda_n]$, outside the
triangles $T_k$ and outside the closed
disks centered in $i\lambda_k$ of radii
$\rho_k$, given by \eqref{eq:circles-m}.

\item If moreover, $m\le n$ and the smallest singular value $\sigma_m$ of $B$
satisfies $\sigma_m > 2\sqrt{2} \Delta$, then  exactly
$m$ eigenvalues of the closed-loop system lie in  box
\[
[-\|B\|,-\frac{\sqrt{6}}{4}\sigma_m]\times[\lambda_1,\lambda_n],
\]
and the other $n-m$ eigenvalues lie in the box
\[
(-\sqrt{3}\Delta, 0)\times[\lambda_1,\lambda_n].
\]

\item In the case $m=1$, the assertion of (\ref{item}) holds for disks
with the same centra and larger radii $\rho^1_k$, instead of
$\rho_k$.
\end{enumerate}
\end{theor}

\

Notice that $\ellest=\min_k\rho_k/\sqrt{2}$,
$\ellest^1=\min_k\rho^1_k / \sqrt{2}$
(see Figure \ref{zeros-area-small}). Therefore
Theorem \ref{thm:main} is an immediate consequence of the above theorem.

\textit{Remarks.}
\begin{enumerate}

\item
Though we only deal with finite dimensional optimal control, we believe that
the lower bounds for the decay rate $\gadecay$, given in Theorem~1, can be extended to well-posed systems with
unbounded skew-symmetric operator $A$. Then, in order to
get a nontrivial estimate, $B$ should be unbounded, but still can
be finite dimensional.
We refer to \cite{Mikk2006} and references therein for a discussion of
exponential stabilization of the closed loop systems obtained by linear quadratic optimization.
For infinite dimensional systems, the choice $Q=I$
(or, more generally, $Q=f(iA)$, where $f$ is a positive function on $\mathbb R$),
is rather natural.

In \cite{CurWeiss}, the same question was discussed for the
collocated feedback $u=-B^*x$, which in many cases
stabilizes the system. This choice of feedback is very
common, for instance, in the control of flexible structures.
In general, the decay rates of the corresponding closed loop
systems are incomparable, and one can give examples when the
collocated feedback yields much lower decay rate
than the linear quadratic optimization.

\item \label{remark-shifted}
It should also be mentioned that (apart from
the pole placement algorithms), there is a standard way
to obtain a closed loop system with a prescribed
decay rate. In application to our case, one has to
fix some shift $\tau>0$ and find a linear quadratic optimal feedback $F$ for the
pair $(A+ \tau I, B)$.
Then the closed loop matrix $A-BF$ will have $\gadecay>\tau$. See, for instance,
\cite[Section 3.5]{AndersMoore}. This method works well only for
small or moderate values of~$\tau$.

For instance, take the $11\times 11$ matrix
$A=i\operatorname{diag}(-5, -4, \dots, 4, 5)$ and the $11\times 1$ column
$B=(1, \dots, 1)^*$.
Let $u(t)=-F_\tau x(t)$ be the feedback obtained by the above procedure,
$A^\tau_{\text{cl.loop}}$ be the corresponding closed look matrix.
Let  $X_\tau$ be given by $J^u(x_0) = x_0^* X_\tau x_0$
(see the Introduction), with $J^u$ given by
\eqref{eq:j-diag}, where $(x(t), u(t))$ is the motion
that corresponds to this feedback.
If no shift is applied to $A$ ($\tau=0$), then
$\gadecay\approx 0.66$ and $\|X_\tau\|\approx 5.49$.
Next, $X_\tau$ has the norm around $1.23\cdot 10^3$
for $\tau=1$ and the norm around $1.62 \cdot 10^6$
for $\tau=2$. The latter choice of the
shift $\tau$ gives a large quadratic cost
functional even if one omits in \eqref{eq:j-diag} the term containing $u(t)$:
the matrix
$X^0_\tau\defin \int_0^\infty \exp(A^{\tau\,*}_{\text{cl.loop}} t)\exp(A^\tau_{\text{cl.loop}} t)\, dt$
has the norm around $4.6\cdot 10^4$ for $\tau=2$.



\end{enumerate}

Before proving Theorem \ref{thm:main-detailed}, we need some preliminaries and several lemmas.

\subsection{The function $\Phi$ and its zeros}
\label{function-phi}

The rational matrix function defined as
\begin{equation}
 \label{eq:phi-problem}
\Phi(z) = I - B^*(zI - A)^{-2}B.
\end{equation}
is important in the control system theory.
It is known that $\Phi$ factorizes as
\[
 \Phi(z) = M(-\overline{z})^*M(z),
\]
where
\[
 M(z) = B^*X(zI-A)^{-1}B + I.
\]
The theory also shows that
\[
 M(z)^{-1} = -B^*X(zI - (A - BB^*X))^{-1}B + I.
\]
See, for instance, the book by Zhou, Doyle and Glover \cite[chapter
13.4]{ZhouDoyGlov} for a proof of this factorization.

Hence,
the eigenvalues of $A - BB^*X$ are poles of $M(z)^{-1}$ in the sense
that if $z_0$ is an eigenvalue of $A - BB^*X$ then $\det (M(z_0)^{-1}) =
\infty$. It follows from the factorization of $\Phi(z)$ that the zeros of
$\Phi(z)$ (in the sense that $\det \Phi(z) = 0$) are
\[
 \{z \in \mathbb{C} : \det\Phi(z) = 0\} = \{z \in \mathbb{C} : z \in
\siclloop\ \textrm{or}\ -\overline{z} \in \siclloop\}.
\]

\begin{definition}
 Let $\Phi(z)$ be as in \eqref{eq:phi-problem} and $z_0 \in \mathbb{C}$ such
that $\det \Phi(z_0) = 0$. If $\Re z_0 < 0$, then $z_0$ is called a \emph{stable
zero} of $\Phi(z)$. If $\Re z_0 > 0$, then $z_0$ is called an \emph{anti-stable
zero} of $\Phi(z)$.
\end{definition}
So the stable zeros of $\Phi(z)$ are exactly the eigenvalues
of $A_{\text{cl.loop}}$.

The function $\Phi$ will be very useful to make estimations of
the cost characteristic $\gamma_{\mathrm{decay}}$. The relation between
$\gamma_{\mathrm{decay}}$ and $\Phi(z)$ is
\begin{equation}
 \gadecay(A, B) = \min \{|\Re z| : \det\Phi(z) = 0\}.
\end{equation}

If we define
\begin{equation}
 \label{eq:f-def}
 f(\lambda; z) = \frac{1}{(z - \lambda)^2},
\end{equation}
then, for fixed $z \in \mathbb{C}$,
$f$ is holomorphic in $\lambda$ on
$\BC\setminus\{z\}$
and hence $f(A; z)$ is well defined if $z \notin
\sigma(A)$. We can write
\[
 \Phi(z) = I - B^*f(A; z)B.
\]
Using \eqref{eq:A-decomposition} and \eqref{eq:decomp-b}, we
get
\begin{equation}
\label{eq:phi-decomp}
\Phi(z) = I - \sum_{j=1}^n\frac{1}{(z-i\lambda_j)^2}b_jb_j^*.
\end{equation}

An important remark is that $\Phi(z)$ is Hermitian and positive along the
imaginary axis where it is defined.
Indeed, we have $A=iA_0$, where $A_0=A_0^*$.
Let $t \in \mathbb{R}$, $t \neq \lambda_j$, then
\[
\Phi(it) = I - B^*(itI - A)^{-2}B = I + B^*(tI - A_0)^{-2}B > I,
\]
because $B^*(tI - A_0)^{-2}B$ is Hermitian and positive.

\begin{lemma}
 \label{zero-location}
The zeros of $\Phi$ lie in the box in the complex plane given by
$|\Re z| \leq \|B\|$, $\lambda_1 \leq \Im z \leq \lambda_n$.
\end{lemma}
\begin{proof}
Recall that the real and imaginary parts of an operator $T$ are defined by
\[
 \Re T = \frac{1}{2}(T + T^*)\qquad \Im T = \frac{1}{2i}(T - T^*).
\]
Put
\[
h(\lambda; z) = \frac{1}{2}(f(\lambda; z) + f(-\lambda;
\overline{z})), \quad
 g(\lambda; z) = \frac{1}{2i}(f(\lambda; z) - f(-\lambda; \overline{z})).
\]
Then $h$ and $g$ are meromorphic in $\lambda$ on the whole plane.
It is easy to see that
\begin{equation}
\label{eq:re-im-phi}
 \Re \Phi(z) = I - B^*h(A; z)B, \qquad \Im \Phi(z) = -B^*g(A; z)B.
\end{equation}

If $z = x + iy$, $\lambda \in \mathbb{R}$, a direct computation shows
that
\begin{equation}
\label{eq:h-computation}
 h(i\lambda; z) = \frac{x^2 - (y - \lambda)^2}{(x^2 + (y -
\lambda)^2)^2},
\end{equation}
\begin{equation}
\label{eq:g-computation}
 g(i\lambda; z) = \frac{-2x(y - \lambda)}{(x^2 + (y -
\lambda)^2)^2}.
\end{equation}

First we show that if $|\Re z| > \|B\|$ then $\Re \Phi(z) > 0$ so $z$
is not a zero of $\Phi(z)$. Let $\xi \in \mathbb{C}^n$ with $\|\xi\| = 1$. Then
\[
 \langle \Re \Phi(z) \xi, \xi \rangle = 1 - \langle h(A;z) B\xi, B\xi \rangle \geq
 1 - \|B\|^2\max_{\lambda \in \sigma(A)}h(\lambda;z).
\]
Now, using \eqref{eq:h-computation}, if $|\Re z| > \|B\|$, it follows
\[
 \max_{\lambda \in i\mathbb{R}} h(\lambda; z) < \frac{1}{\|B\|^2}
\]
and therefore $\langle \Re \Phi(z) \xi, \xi \rangle > 0$ for these $z$.

Now observe that if either $\Im z < \lambda_1$  or $\Im z > \lambda_n$ then
\eqref{eq:g-computation} shows that $g(\lambda; z)$ has constant sign for all
$\lambda \in \sigma(A)$ and therefore $\Im \Phi(z)$ is either postive or
negative (since we may assume $\ker B = 0$) so that $z$ is not a zero of
$\Phi(z)$.
\end{proof}

In Lemma \ref{zero-location} we have seen that the zeros of $\Phi(z)$ cannot
be too far from the imaginary axis. The next two lemmas imply that
the zeros cannot be too close to the imaginary axis.

\begin{lemma}
\label{triangles}
Define the angles
\[
\mathfrak{A}_k=\big\{z\in \BC_-:  \frac{3\pi}{4}< \arg (z-i\la_k) <
\frac{5\pi}{4}\big\}, \qquad 1\le k\le n.
\]
If $z$ is in the left half-plane, but does not belong to the union of these
angles, then $\Re \Phi(z)\ge I$.
\end{lemma}

\begin{proof}
Put $z=x+iy$. It follows from the hypothesis on $z$ that
\[
x^2 - (y - \lambda_j)^2 \leq 0, \qquad j=1,\dots, n.
\]
Defining $h(\lambda; z)$ as in the proof of Lemma \ref{zero-location} and using
\eqref{eq:h-computation}, we get that $h(A; z) \leq 0$, so that $\Re \Phi(z)
\geq I$.
\end{proof}

It follows from the above two lemmas that the stable zeros of $\Phi$ lie
in the band $\lambda_1 \leq \Im z \leq \lambda_n$ and
outside the triangles $T_1, \dots, T_{n-1}$.

\begin{lemma}
\label{circles-m}
$\Phi$ has no zeros in the disks $\overline{D(i\la_k, \rho_k)}$, $1\le k\le n$.
\end{lemma}

\begin{proof}
If $\rho_k = 0$ then the lemma is vacuously true for the corresponding $k$.
Hence, assume $\rho_k > 0$.
Suppose $\Re z < 0$ and $|z - i\lambda_k| \leq \rho_k$ for
some $k$. Fix this index $k$.
Observe that $\frac{\delta_k}{\sqrt{2}}$ is the length of
the legs of one of the triangles $T_\ell$, whose vertex is in
$i\la_k$. Since
\[
 \rho_k^2 = \frac{\|b_k\|^2}{\big(1 + 2\frac{\|B\|^2}{\delta_k^2}\big)^2} <
\frac{\|b_k\|^2\delta_k^2}{2\|B\|^2} \leq \frac{\delta_k^2}{2},
\]
it follows that $z$ belongs to $D(i\la_k,\frac{\delta_k}{\sqrt{2}})$.

It can be shown geometrically that the intersection of
$\overline{D(i\la_k,\frac{\delta_k}{\sqrt{2}})}$ with
$\overline{D(i\la_j,\frac{\delta_k}{\sqrt{2}})}$ ($j\ne k$; notice that we take
the same radii)
and with the left half-plane is either empty or is contained in one of the triangles $T_\ell$.
Therefore, if $z\in \overline{D(i\la_j,\frac{\delta_k}{\sqrt{2}})}$ for some
$j\ne k$, then $z$ is inside
of one of the triangles, and it has already been shown that then $z$ will not be a zero of $\Phi$.

So let us assume that
\begin{align}
\label{x-not-in-Dj}
|z-i\la_k|<\frac{\delta_k}{\sqrt{2}}\, , \quad \text{but} \enspace
|z-i\la_j| > \frac{\delta_k}{\sqrt{2}} \enspace \text{for} \enspace j\ne k.
\end{align}
Put
\[
 C(z) = \Phi(z) + \frac{1}{(z-i\lambda_k)^2} b_k b_k^* =
 I - B^*\bigg(\sum_{j\neq k} \frac{1}{(z-i\lambda_j)^2}v_jv_j^*\bigg) B
\]
(see \eqref{eq:A-decomposition}).
It follows from the first inequality in \eqref{x-not-in-Dj} that
$z\notin\overline{\mathfrak{A}}_j$
for $j\ne k$.
Hence Lemma  \ref{triangles}, applied to the configuration of $n-1$ points
$\{i\la_1, i\la_2, \dots, i\la_n\}\setminus \{i\la_k\}$ on the imaginary axis,
gives that $\Re C(z) > I$ (recall we have assumed $\ker B = 0$).

By \eqref{x-not-in-Dj}, we also have
\[
 \|C(z)\| \leq M\defin 1 + \frac{2\|B\|^2}{\delta_k^2}.
\]
Next, let us show that the above properties $\|C(z)\| \leq M$ and $\Re C(z) > I$, imply that $\Re C(z)^{-1} > \frac{1}{M^2}$.
Indeed, take any $\xi \in \mathbb{C}^n$ with $\|\xi\|=1$ and set $\eta = C^{-1}(z)\xi$.
Then $1 = \|\xi\| \leq \|C(z)\| \| \eta \|$ so that $\|\eta \| \geq \frac{1}{M}$. Hence,
\[
 \Re \langle C(z)^{-1} \xi\, , \xi \rangle =
 \Re \langle C(z)\eta, \eta \rangle >
 \|\eta\|^2 \geq \frac{1}{M^2}\,   ,
\]
and the inequality $\Re C(z)^{-1} > \frac{1}{M^2}$ follows.

Now suppose that $\Phi(z) \xi = 0$ for some fixed $\xi\in \mathbb{C}^m$, $\xi
\neq
0$. Then,
\[
 \big(C(z) - \frac{1}{(z - i\lambda_k)^2}b_kb_k^*\big) \xi = 0.
\]
Since $C(z)$ is invertible, we have $b_k^* \xi \neq 0$. Multiply the above
equality by $(b_k^* \xi )^{-1} b_k^*C(z)^{-1}$ and regroup terms to yield
\[
 b_k^*C(z)^{-1}b_k = (z - i\lambda_k)^2.
\]
Since $\Re C(z)^{-1} > \frac{1}{M^2}$, we get
\[
 |z - i \lambda_k|^2 = |b_k^*C(z)^{-1}b_k| > \frac{\|b_k\|^2}{M^2}= \rho_k^2.
\]
Therefore $|z -i \la_k| > \rho_k$, a contradiction.
\end{proof}

In the case $m = 1$, the above lemma can be strengthened.

\begin{lemma}
\label{circles}
If $m = 1$, then $\Phi$ has no zeros in the disks
$\overline{D(i\la_k,\rho^1_k)}$.
\end{lemma}
\begin{proof}
Assume that $m=1$, $\Phi(z) = 0$ and for some index $k$,
$z\in \overline{D(i\la_k,\rho^1_k)}$.
Proceed as in the previous lemma to deduce that
$|z - i\lambda_j| > \frac{\delta_k}{\sqrt{2}}$ for $j
\neq k$. Now, since $\Phi(z) = 0$, we have
\[
 \frac{|b_k|^2}{(z - i\lambda_k)^2} = 1 -
\sum_{j\neq k} \frac{|b_j|^2}{(z - i\lambda_j)^2}
\]
(notice that now $b_j$ are complex numbers). It follows that
\[
 \frac{|b_k|^2}{|z - i\lambda_k|^2} \leq
 1 + \sum_{j\neq k} \frac{|b_j|^2}{|z - i\lambda_j|^2} <
1 + \sum_{j=1}^n \frac{2|b_j|^2}{\delta_k^2} =
1 + 2\,\frac{\|B\|^2}{\delta_k^2},
\]
so that $|z - i\lambda_k| > \rho^1_k$, a contradiction.
\end{proof}

\begin{lemma}
\label{gaps}
Let $\sigma_m$
be the minimum singular value of $B$. If $\sqrt{3}\Delta
< \frac{\sqrt{6}}{4}\sigma_m$, then exactly $m$ of the stable zeros of
$\Phi(z)$ lie in the box given by
\[
 -\|B\| \leq \Re z \leq -\frac{\sqrt{6}}{4}\sigma_m, \quad \lambda_1 \leq \Im z
\leq \lambda_n,
\]
and the $n - m$ remaining stable zeros lie all in the box given by
\[
 -\sqrt{3}\Delta < \Re z < 0, \quad \lambda_1 \leq \Im z \leq \lambda_n.
\]
In particular, no stable zero lies in the band $\Re z \in
(-\frac{\sqrt{6}}{4}\sigma_m,-\sqrt{3}\Delta]$.
\end{lemma}
\begin{proof}
 The restriction to $\lambda_1 \leq \Im z \leq \lambda_n$ comes
 from Lemma
\ref{zero-location}. To prove the statement about boxes, suppose
that $\sigma_m$ satisfies the hypothesis given.

Let $\Gamma_c:[\al,\be]\to \BC$ be the closed positively oriented contour, traversing the
boundary of the box $[-c,c]\times[-d,d]$. Since $\Phi(\infty) = I$,
$d$ can be chosen large enough so that all the eigenvalues of $\Phi(z)$
are arbitrarily close to $1$ when $z$ is on the horizontal
segments of $\Gamma_c$. We assume that
$\Gamma_c(\al)=\Gamma_c(\be)=-c-id$.

Let $\gamma_c$ be the right vertical segment of $\Gamma_c$, going
from $c-id$ to $c+id$. We subdivide $\gamma_c$ into three segments,
\[
[c - id, c + i\la_1],
\quad
[c + i\la_1,c + i\la_n],
\quad
[c + i\la_n, c + id].
\]
We will use expressions \eqref{eq:re-im-phi} for
the real and imaginary parts of $\Phi(z)$.
First observe that if $z \in [c - id, c + i\la_1]$, then
 $\Im \Phi(z)< 0$. Indeed, for these $z$,
 $\Re z > 0$ and $\Im z < \lambda_j$ for all $j$. It follows that
$g(i\lambda_j; z) > 0$ and therefore $\Im \Phi(z) < 0$
(see \eqref{eq:g-computation}). Hence, all the eigenvalues
of $\Phi(z)$ lie in the open lower half-plane.

Similarly, if $z \in [c + i\la_n, c + id]$,
one has $\Im \Phi(z) > 0$. Hence for these $z$, all the
eigenvalues of $\Phi(z)$ lie in the upper half-plane.

Now we will show that if
\begin{equation}
 \label{eq:c-conditions}
 \sqrt{3}\Delta \le c < \frac{\sqrt{6}}{4} \sigma_m
\end{equation}
then $\Re \Phi(z) < 0$ for  $z \in [c + i\la_1,c + i\la_n]$.
Write $z = c + iy$, $y \in [\lambda_1,\lambda_n]$. Then, using
\eqref{eq:h-computation} and \eqref{eq:c-conditions},
we get that for all $j$,
\[
 h(i\lambda_j; z) = \frac{c^2 - (y - \lambda_j)^2}{(c^2 +
(y - \lambda_j)^2)^2} \geq \frac{c^2 - \Delta^2}{(c^2 + \Delta^2)^2} \geq
\frac{c^2 - c^2/3}{(c^2 + c^2/3)^2} =
\frac{3}{8c^2} > \frac{1}{\sigma_m^2}.
\]

If $\xi \in \BC^m$ with $\|\xi\|=1$, then
\[
 \left\langle \Re \Phi(z) \xi, \xi \right\rangle =
 1 - \big\langle h(A; z) B\xi,\, B\xi \big\rangle
\leq
1 - \big(\min_{\lambda\in\sigma(A)} h(\la; z) \big)\cdot
\|B\xi\|^2
< 1 - \frac{1}{\sigma_m^2}\|B \xi\|^2 \leq 0,
\]
because $\|B\xi\| \geq \sigma_m\|\xi\|$. Hence, $\Re \Phi(z) < 0$.

Since $\Phi(-\overline{z}) = \Phi(z)^*$, $\Phi(z)$ behaves similarly
on the left vertical segment of $\Gamma_c$.

Now choose $c$ satisfying \eqref{eq:c-conditions} and study the winding number
of $\det \Phi \circ \Ga_c$ around $0$. The $m$ eigenvalues of $\Phi(z)$,
$\phi_1(z),\ldots,\phi_m(z)$,
can be numbered so that $\phi_j\circ \Ga_c(t)$ are all
continuous functions of the parameter $t, t\in [\al,\be]$.
Since
$\det \Phi\circ \Gamma_c = (\phi_1\circ \Gamma_c)\cdot (\phi_2\circ \Gamma_c) \cdot \,\ldots\, \cdot (\phi_m\circ \Gamma_c)$,
it follows that
\begin{equation}
\label{det-Phi}
 \operatorname{index} (\det \Phi\circ \Gamma_c) = \sum_{j=1}^m \operatorname{index}(
\phi_j\circ \Gamma_c).
\end{equation}
Let us calculate the winding number of the curves $\phi_j \circ \Gamma_c$.

When $z=\Ga_c(t)$ is in the lower horizontal segment, $\phi_j$ are all close to $1$. Then,
as $z$ travels through $\gamma_c$, $\phi_j$ first are all in the lower half-plane,
then go to the left half-plane and then to the upper half-plane.
When $z$ is in the upper horizontal segment, all the
numbers $\phi_j$ are again close to $1$.
It follows that by choosing $d$ sufficiently large, we can make the winding number of each of the
functions $\phi_j$ to be arbitrarily close to $-1$ on each of two vertical segments of $\Gamma_c$ and
arbitrarily close to $0$ on the two horizontal parts of $\Gamma_c$. Since
$\det \Phi \circ \Gamma_c:[\al, \be]\to \BC$ is a closed curve, its
winding number around $0$ is
an integer. By \eqref{det-Phi}, it is equal to $-2m$.
Using the argument principle and
the fact that $\det \Phi(z)$ has $2n$ poles counting multiplicities inside
$\Gamma_c$, one gets that $\det \Phi(z)$ has $2n -
2m$ zeros inside $\Gamma_c$. Hence, $\Phi(z)$ has $n-m$ stable zeros inside
$\Gamma_c$.

Setting $c = \sqrt{3}\Delta$, we see that there are $n-m$ stable zeros inside
$(-\sqrt{3}\Delta,0)\times[\lambda_1,\lambda_n]$. Letting $c \to
\frac{\sqrt{6}}{4}\sigma_m$, we obtain that again
$\Phi$ has $n-m$ stable zeros inside the box
$(-\frac{\sqrt{6}}{4}\sigma_m, 0)\times[\lambda_1, \lambda_n]$. The remaining
$m$ stable zeros must lie all outside this box, and by Lemma \ref{zero-location},
they belong to the box
to $[-\|B\|,-\frac{\sqrt{6}}{4}\sigma_m]\times[\lambda_1,\lambda_n]$.
\end{proof}

\begin{proof}[Proof of Theorem \ref{thm:main-detailed}]
All the statements of this  theorem follow from Lemmas
\ref{zero-location}--\ref{gaps}.
\end{proof}

\begin{proof}[Proof of Theorem \ref{thm:main}]
As we already pointed out just after the statement of Theorem
\ref{thm:main-detailed}, Theorem \ref{thm:main} is its direct consequence.
\end{proof}

Using Theorem \ref{thm:main-detailed}, we can provide upper
bounds for the value of $\gadecay$.

\begin{corol}
\label{gamma-bound}
 The following upper bound always holds for the value of
$\gamma_{\mathrm{decay}}$,
\begin{equation}
 \gamma_{\mathrm{decay}}(A, B) \leq \|B\|.
\end{equation}

If in addition, $\sigma_m$, the smallest singular value of $B$,  satisfies
$\sigma_m > 2\sqrt{2}\Delta$ and $m < n$, then
\begin{equation}
 \gamma_{\mathrm{decay}}(A, B) < \sqrt{3}\Delta.
\end{equation}
If $m=1$, then $\gadecay \leq 2\sqrt{2}\, \Delta$ for any $B$ such that the pair
$(A,B)$ is controllable.
\end{corol}
\begin{proof}
 The first bound comes from lemma \ref{zero-location}. Under the conditions of
the second bound, using lemma \ref{gaps} it follows that $\Phi(z)$ has at
least one zero on $(-\sqrt{3}\Delta, 0)\times[\lambda_1,\lambda_n]$.
If $m=1$, then $\si_m=\|B\|$, which gives the last statement.
\end{proof}

\smallskip
\textit{Remark.}
Upper and lower bounds for $\gadecay$ given in Theorem 1 and in the above Corollary
fail for a general (not skew-Hermitian) $A$ with imaginary spectrum. Consider, for
instance, matrices
\[
A_2=i\,
\begin{pmatrix}
        -1 & t\\
       0 & 1
\end{pmatrix},
\qquad
A_3=i\,
\begin{pmatrix}
        -1 & t & 0 \\
        0 &  1 & 0 \\
        0 &  0 & 0
\end{pmatrix}
\]
(so that $\sigma(A_2)=\{-i, i\}$ and $\sigma(A_3)=\{-i, 0, i\}$
for all $t\in \mathbb{R}$). Put $B_2=(1, 1)^T$, $B_3=(1, 1, 1)^T$.
Then numerical simulation shows that for large positive $t$'s,
$\gadecay(A_2, B_2)$ is very large
(and does not satisfy $\gadecay\le \|B\|$) and
$\gadecay(A_3, B_3)$ is very close to zero
(and does not satisfy $\gadecay\ge \ellest^1$).
In fact, the simulation suggests that $\gadecay(A_2, B_2)\to +\infty $ and $\gadecay(A_3, B_3)\to 0 $
as $t\to +\infty $.

\medskip

As we already mentioned before,
Theorem \ref{thm:main} also implies lower estimates
for $\gamma_{\mathrm{decay}}$, namely
$\gamma_{\mathrm{decay}}> \ellest$ and
$\gamma_{\mathrm{decay}}> \ellest^1$ for $m=1$.

The following upper bound holds for $\ellest$:
\[
\ellest=\min_{1\leq k \leq n}
\frac{\|b_k\|}{\sqrt{2}\,(1 + 2\frac{\|B\|^2}{\delta_k^2})}
\leq
\frac{\sqrt{2}}{4}\sqrt{\frac{m}{n}}\frac{\delta^2}{\|B\|}.
\]
Indeed, the inequality
\[
\|b_k\| \leq \sqrt{\frac{m}{n}}\|B\|,
\]
(see the proof of Theorem \ref{limit-case}) implies that
we get
\[
\ellest \leq \frac{\sqrt{2}}{4} \min_k \frac{\|b_k\| \delta_k^2}{\|B\|^2} \leq
\frac{\sqrt{2}}{4}\sqrt{\frac{m}{n}}\frac{\delta^2}{\|B\|}.
\]
Similarly,
$\ellest^1\leq \de /(2\sqrt{n})$.
One can guess that a matrix $B$ in which all $\|b_k\|$ are as big as possible can be
used to ensure a nearly optimal stabilization of the system. A matrix
with these characteristics will be given below in Theorem \ref{limit-case}.

\section{The estimate of $\gadecay$ in the case of a sufficient
separation of the open-loop spectrum}
\label{sec:enough-sep}

Here we will assume that the minimal separation $\delta(A)$, defined in
\eqref{eq:delta}, is rather big
in comparison with $\|B\|$.
We will use the following analogue of Rouch\'e's theorem for
matrix-valued functions.

\begin{lemma}
Let $F(z)$ and $G(z)$ be meromorphic functions on some
open subset $D \subset \mathbb{C}$, whose values are $m\times m$ complex matrices. Let $\gamma$ be a closed curve
in $D$ such that $\det F(z)$ has no poles or zeros on $\ga$.
If $\|F^{-1}(z)G(z)\| < 1$ for all $z \in \gamma$, then the
scalar functions $\det F(z)$ and $\det (F(z) + G(z))$ have the same winding number along
$\gamma$.
\end{lemma}
This lemma is known. See, for instance, \cite[Theorem 2.2]{GohgSigal} (for
operator-valued functions)
or \cite{Monden-Arim}.

The next theorem locates the points of the closed-loop spectrum inside
disks of radii $r_k$ such that $r_k \to 0$ as $\frac{\|B\|^3}{\delta^2} \to
0$. Recall that the zeros of $\Phi$ which lie in the left
half-plane coincide with the eigenvalues of the closed-loop system (see
Subsection~\ref{function-phi}).

\begin{theor}
\label{zeros-disks}
 Suppose $k$ is an index such that $\frac{\|B\|^2}{\delta_k^2} <
\frac{\left(2-\sqrt{2}\right)^2}{2}$. Then $\Phi$ has exactly one zero in the
open disk of centre $z_k = -\|b_k\| + i\lambda_k$ and radius $r_k
= \frac{2}{\left(2-\sqrt{2}\right)^2}\frac{\|B\|^2}{\delta_k^2}\|b_k\|$.
\end{theor}

\begin{proof}
Observe that $r_k < \|b_k\|$ and consider the contour
\[
 \gamma = \{ z_k + r_k e^{i\theta} : 0 \leq \theta \leq 2\pi\}
\]
and the functions
\[
 F(z) = I - \frac{1}{(z-i\lambda_k)^2}b_kb_k^*
\]
\[
 G(z) = \Phi(z) - F(z).
\]
These functions are holomorphic on $\gamma$ and its interior. We will
prove that $\|G(z)\| < \|F(z)^{-1}\|^{-1}$ for $z \in \gamma$, so
that we can use the above version of Rouch\'e's theorem.

First observe that $F(z)$ is normal, so that $\|F(z)^{-1}\|^{-1} =
\min_{\lambda\in\sigma(F(z))} |\lambda|$. The spectrum of $F(z)$ can be
computed easily:
\[
 \sigma(F(z)) = \left\{1, 1 - \frac{\|b_k\|^2}{(z-i\lambda_k)^2}\right\}.
\]

Take any $z$ such that $|z-z_k|=r_k$ and put
$z = z_k + r_k e^{i\theta}$.
Notice that $r_k < \|b_k\|$ implies
$\big|r_k e^{i\theta} - 2\|b_k\|\big| \geq \big| \|b_k\| - r_k e^{i\theta} \big|$.
We get
\[
\begin{split}
\left|1 - \frac{\|b_k\|^2}{(z-i\lambda_k)^2}\right|         &=
\bigg|1 - \frac{\|b_k\|^2}{(\|b_k\| - r_k e^{i\theta})^2}\bigg| =
 \left|\frac{r_k^2 e^{2i\theta} - 2r_k \|b_k\| e^{i\theta}}{(\|b_k\| - r_k
e^{i\theta})^2}\right|
\\
&= \frac{r_k \big|r_k e^{i\theta} - 2\|b_k\|\big|}
{\big|\|b_k\| - r_k e^{i\theta}\big|^2}
 \geq \frac{r_k}{\||b_k\| - r_k e^{i\theta}|} \geq \frac{r_k}{\|b_k\| + r_k}
\geq
\frac{r_k}{2\|b_k\|},
\end{split}
\]
and so $\|F(z)^{-1}\|^{-1} \geq \frac{r_k}{2\|b_k\|}$.
Hence it will suffice to prove that $\|G(z)\| < \frac{r_k}{2\|b_k\|}$.
By \eqref{eq:phi-decomp},
\[
 G(z) = -B^*\bigg(\sum_{j\neq k} \frac{1}{(z-i\lambda_j)^2}v_jv_j^*\bigg) B.
\]
Then, observe that the condition $\frac{\|B\|^2}{\delta_k^2} <
\frac{\left(2-\sqrt{2}\right)^2}{2}$ implies $r_k < \left(\sqrt{2} -
1\right)\delta_k$ since
\[
 r_k < \|b_k\| \leq \|B\| < \frac{2 - \sqrt{2}}{\sqrt{2}}\delta_k =
\big(\sqrt{2} - 1\big) \delta_k.
\]
Now we have for $z \in \gamma$ and for all $j \neq k$
\begin{multline*}
|z - i\lambda_j| = \big|r_k e^{i\theta} - \|b_k\| + i\lambda_k - i\lambda_j\big|
\geq
\big|i\lambda_k - i\lambda_j - \|b_k\|\big| - r_k \\
\geq
|\lambda_k - \lambda_j| - r_k
 > \de_k-(\sqrt{2}-1)\de_k=\big(2 - \sqrt{2}\big)\de_k.
\end{multline*}

Hence it follows that
\[
 \|G(z)\|
 < \frac{\|B\|^2}{\left(2-\sqrt{2}\right)^2\delta_k^2}
= \frac{r_k}{2\|b_k\|}
\le \|F(z)^{-1}\|^{-1}.
\]
So, by Rouch\'e's theorem, $\det F(z)$ and $\det (F(z) +
G(z))$ have the same number of
zeros inside $\gamma$. The only zeros of $\det F(z)$ are $z_k$ and
$-\bar z_k$.
Since $\gamma$ lies completely in the left half-plane,
$\det F(z)$ has exactly one zero inside $\gamma$.
Therefore
$\det \Phi(z)=\det (F(z) + G(z))$ has exactly one zero inside $\ga$.
\end{proof}

\begin{corol}
\label{corol-phiks}
Set
\begin{equation}
\label{eq:defin-phik}
\varphi_k \defin \frac{2\|B\|^2}{(2-\sqrt{2})^2\delta_k^2}, \qquad 1\le k \le n.
\end{equation}
Suppose $\varphi_k < 1$ for at least one index $k$.
Put
\[
\Gamma_+ = \min \{(1+\varphi_k)\|b_k\| : \quad \varphi_k < 1\}, \qquad
\Gamma_- = \min_k (1-\varphi_k)\|b_k\|.
\]
Then $\gadecay < \Gamma_+$.
If moreover $\varphi_k < 1$ for all $k$, then
\[
\Gamma_- < \gadecay < \Gamma_+.
\]
\end{corol}
\begin{proof}
If $\varphi_k < 1$ for some $k$, the preceding theorem shows that some
eigenvalue $\nu$ of
the closed-loop system satisfies $-\Re \nu < (1+\varphi_k)\|b_k\|$,
and the upper bound follows. If $\varphi_k < 1$ for all $k$ and $\nu$ is any
eigenvalue
of the closed-loop system, then $-\Re \nu > (1-\varphi_k)\|b_k\|$ for some
$k$, so that the lower bound follows.
\end{proof}

Using Theorem \ref{zeros-disks},
when $\delta(A)$ is sufficiently large, we can give a matrix $B$, in a sense close to optimal.

\begin{theor}
\label{limit-case} Suppose $m\le n$.
 Let $w$ be the primitive $n$-th root of $1$ given by
\[
 w = e^{-i\frac{2\pi}{n}}.
\]
Let $\beta > 0$.

Let the matrix $\widehat{B}$ be represented in the orthonormal basis given by
$\{v_j\}$, the eigenvectors of $A$, as
\begin{equation}
\label{eq:optimal-B}
 \widehat{B} = \frac{\beta}{\sqrt{n}}\left[\begin{array}{cccc}
                  w^{0\cdot0} & w^{0\cdot 1} & \cdots & w^{0\cdot(m-1)} \\
          w^{1\cdot0} & w^{1\cdot1} & \cdots & w^{1\cdot(m-1)} \\
          \vdots & \vdots & \ddots & \vdots \\
          w^{(n-1)\cdot0} & w^{(n-1)\cdot1} & \cdots &
w^{(n-1)\cdot(m-1)} \\
                 \end{array}\right].
\end{equation}

 Then, $\|\widehat{B}\| = \beta$ and for any $\varepsilon > 0$, there exists $K > 0$
such that if $\delta(A) > K$,
then
\[
 \big(\sup_{\|B\|=\beta} \gamma_{\mathrm{decay}}(A,B) \big) - \gamma_{\mathrm{decay}}(A,
\widehat{B}) < \varepsilon.
\]
\end{theor}
\begin{proof}
First observe that $\widehat{B}$ is related to the unitary Discrete Fourier
Transform. If $U \in \mathbb{C}^{n \times n}$ is the matrix of the unitary DFT,
then
\[
 \widehat{B} = \beta U \left[\begin{array}{c} I_{m\times m} \\ \hline 0_{n-m \times
m} \end{array}\right].
\]
It follows that $\|\widehat B\| = \beta$. Define
$\widehat b_j$
in the same way as in
as in \eqref{eq:def-bj},
that is, put $\widehat b_j=\widehat B^* v_j$.
Then $\|\widehat b_j\| =
\beta\sqrt{\vphantom{|}\frac{m}{n}}$ for all $j = 1,\ldots,n$.

Let $B$ be arbitrary with $\|B\|=\beta$.
Let $\varepsilon>0$ be given. Define $K>0$ from
\[
 K^2 = \max\left\{\frac{2}{(2 - \sqrt{2})^2}\,\beta^2, \frac{4}{(2 -
\sqrt{2})^2}\,
\frac{\beta^3}{\varepsilon}\right\}.
\]
Suppose $\delta^2 > K^2$. Then $\frac{\beta^2}{\delta^2} <
\frac{(2-\sqrt{2})^2}{2}$, and the hypothesis of
Theorem \ref{zeros-disks} is satisfied for
any index $k$. We obtain disks of radii $r_k$ such that
the zeros of $\Phi$ lie inside this disks. Now,
\[
r_k = \frac{2}{(2-\sqrt{2})^2} \frac{\beta^2}{\delta_k^2}\|b_k\|
\leq
\frac{2}{(2-\sqrt{2})^2} \frac{\beta^3}{\delta^2} < \frac{2}{(2-\sqrt{2})^2}
\frac{\beta^3}{K^2} \leq \frac{\varepsilon}{2}.
\]
Since there is a zero of $\Phi$ in each of these disks of
centre $z_k = -\|b_k\| + i\lambda_k$, we have
\[
 \min \|b_k\| - \frac{\varepsilon}{2} <
\gadecay(A,B) <
\min \|b_k\| +
\frac{\varepsilon}{2}.
\]

Notice that
\[
 \sum_{j=1}^n \|b_j\|^2 = \|B\|_F^2 \leq m\|B\|^2 = m\beta^2.
\]
Hence, $\min \|b_j\|^2 \leq \beta^2\frac{m}{n}$.
Therefore
\[
 \sup_{\|B\| = \beta} \gamma_{\mathrm{decay}}(A,B) \leq \beta\sqrt{\frac{m}{n}}
+
\frac{\varepsilon}{2}.
\]

 Since $\widehat{B}$ has $\|\hat{b_j}\| = \beta\sqrt{ \frac{\vphantom t m}{n}}$ for all $j$,
\[
 \gamma_{\mathrm{decay}}(A, \widehat{B}) > \beta\sqrt{\frac{m}{n}} -
\frac{\varepsilon}{2}
\]
and the theorem follows.
\end{proof}

\begin{corol}
\label{limit-corol}
Assume that $m\le n$.  Let $\beta > 0$. If $A$ is such that
\begin{equation}
 \delta^2(A) > \frac{6}{(2-\sqrt{2})^2}\,\sqrt{\frac{n}{m}}\,\beta^2,
\end{equation}
then the matrix $\widehat{B}$ given in \eqref{eq:optimal-B} satisfies
\begin{equation}
 \gamma_{\mathrm{decay}}(A, \widehat{B}) > \frac{1}{2} \sup_{\|B\| = \beta}
\gamma_{\mathrm{decay}}(A, B).
\end{equation}
\end{corol}

\begin{proof}
By applying Theorem \ref{limit-case}
and its proof
with $\varepsilon =
\frac{2}{3}\beta\sqrt{\frac{m}{n}}$, one gets
\[
 K^2 = \frac{6}{(2-\sqrt{2})^2}\,\sqrt{\frac{n}{m}}\beta^2.
\]

Therefore, if $\delta^2(A) > K^2$, then
\[
 \gamma_{\mathrm{decay}}(A, \widehat{B}) > \beta\sqrt{\frac{m}{n}} -
\frac{\varepsilon}{2}
= \frac{2}{3}\beta\sqrt{\frac{m}{n}} =
\frac{1}{2}\left(\beta\sqrt{\frac{m}{n}} + \frac{\varepsilon}{2}\right) \geq
\frac{1}{2} \sup_{\|B\| = \beta} \gamma_{\mathrm{decay}}(A,B).\qedhere
\]
\end{proof}

\section{Estimates of decay rate in terms of $\Delta_m$ for $m>1$}
\label{sec:delta-m}

We begin with the following remark.
Let $\tht > 0$. Consider the following function
$$
f(\si)=\frac {\si-\tht}{\si+\si^2+\dots+\si^{m+1}},
$$
which is positive on $(\tht, +\infty)$ and
vanishes at $\tht$ and at $+\infty$. We denote by
$\si_0(\tht)$ the point in $(\tht, +\infty)$ where $f$ takes its maximal value
and by $\mu(\tht)=f\big(\si_0(\tht)\big)$ this maximal value.
If we put
$$
P(\si)=
\si+\si^2+\dots+\si^{m+1}-\big(\sum_{j=0}^m (j+1)\si^j\big)\cdot(\si-\tht)
$$
then it is easy to see that
$\si_0(\tht)$ is the unique root of $P$ in $(\tht, +\infty)$
(notice that $P'<0$ on $(\tht, +\infty)$).

In what follows, $P_{\lin\{w_1,\ldots,w_r\}}$ will stand for the orthogonal
projection onto the linear span generated by vectors $w_1,\ldots,w_r$.

\begin{theor}
\label{decay-gener-m}
Let $2\le m<n$. Let $v_j$ be the eigenvectors of $A$ {\upshape(}see
\eqref{eq:la-order} and \eqref{eq:v-j}{\upshape)}. Put
$$
\Pi_k = P_{\lin\{v_k,\ldots,v_{k+m-1}\}}:\BC^n\to \BC^n, \qquad
k=1, \dots, n-m+1.
$$
Assume that positive constants $\ga$, $K$ satisfy
\begin{itemize}

\item[(i)] $\|B\|\le K$;

\item[(ii)] $\|B^*\Pi_kd\|\ge \ga \|\Pi_k d\|$ for all $d \in \BC^n$ and for
all $k$.

\end{itemize}
Define $\Delta_m$ from \eqref{Delta-m}.
Put
$
\si_0=\si_0\big(\frac {K^2}{\ga^2}\big),
$
\begin{equation}
\label{def-rho}
\rho=\frac 12 \min \bigg\{\De_m  \Big(2\, \frac {\si_0^{m+1}-1}
{\si_0-1} -1\Big)^{-1/2}, \;
\ga\Big[\mu\Big(\frac{K^2}{\ga^2}\Big)\Big]^{1/2}
\bigg\} \; .
\end{equation}
Then
all eigenvalues of the closed-loop system lie in the
half-plane $\Re z \leq -\rho<0$.
\end{theor}

Notice that $\De_m$ can be positive even in the case when some of the
eigenvalues of $A$ coincide. We do not exclude this case.

The rest of this Section is devoted to the proof of
this theorem.

\

The plan of the proof is as follows.
First we remark that it is easy to get from (i) and (ii) that $K/\ga \ge 1$.
Hence $\si_0>1$.

Fix some
$z=x+iy\in \BC_-$ such that $-\rho< x<0$.
We have to prove that $z\notin \siclloop$.
To do that, let us  consider a
reordering $\la_{\tau(1)}, \la_{\tau(2)},\dots, \la_{\tau(n)}$ of
the eigenvalues $\la_1, \dots, \la_n$ of $\frac 1 i A$ such that
\begin{equation}
\label{reord}
|z-i\la_{\tau(1)}|\le
|z-i\la_{\tau(2)}|\le
\dots
\le
|z-i\la_{\tau(n)}|.
\end{equation}
Let us assume that
\begin{align}
\label{***}
\big|z-i\la_{\tau(1)}\big|^2\le 2x^2
\end{align}
(if it is not true, then $z\notin \siclloop$, due to Theorem
\ref{thm:main-detailed}). We will divide the spectrum $\sigma(A)$ into
two parts:
$$
\sigma_\cl(A)=
\big\{
i\la_{\tau(1)}, \dots, i\la_{\tau(s)}
\big\},
\quad
\sigma_\far(A)=
\big\{
i\la_{\tau(s+1)}, \dots, i\la_{\tau(n)}
\big\},
$$
where the index $s$ will be elected according to Lemma \ref{lm:suff_separ}
below. (Notice that the reordering \eqref{reord} and this partitioning of
$\sigma(A)$ depend on the position of $z$.)
Once this partition is chosen, we put
$$
\eta_\cl=|z-i\la_{\tau(s)}|^2,
\qquad
\eta_\far=|z-i\la_{\tau(s+1)}|^2.
$$
Introduce the notation
$$
\om=\ga^2,
\qquad
\kappa=K^2
$$
(so that $\om \le \kappa$).

We will say that
\textit{
$\sigma_\cl(A)$ and
$\sigma_\far(A)$ are sufficiently separated} (with respect to $z$) if
\begin{equation}
\label{eq:suff-sep}
\eta_\far > 2 x^2
\quad \text{and} \quad
\eta_\cl <
\om
\frac
{\eta_\far-2 x^2}
{\eta_\far-2 x^2+\kappa}.
\end{equation}

Inequality
$\eta_\far > 2 x^2$
implies that
$
\eta_\far
\big(
\om-\kappa +2x^2-\eta_\far
\big) <0.
$
By using the second inequality in \eqref{eq:suff-sep}, one gets that
the sufficient separation implies the strict inequality
$
\eta_\cl < \eta_\far
$.

Before finishing the proof, we need three lemmas.

\begin{lemma}
\label{lm:suff_separ} For any $z$ such that $-\rho < \Re z<0$ and
\eqref{***} holds there exists an index $s$, $1\le s\le m$, such
that the corresponding parts $\sigma_\far(A)$ and $\sigma_\cl(A)$
of the spectrum of $A$ are sufficiently separated.
\end{lemma}

\begin{proof}
Take some $z$ that satisfies the hypotheses. Let
$\la_{\tau(1)}, \dots, \la_{\tau(t)}$ be all point of the spectrum of $A$
that satisfy
$$
\eta_j \defin |z-i\la_{\tau(j)}|^2\le 2x^2.
$$
Since
$$
|x| < \rho \le \De_m/2
$$
(see \eqref{def-rho}),
it follows that $1\le t \le m$ and that $\eta_{t+1}> 2x^2$.
Assume that the subdivision of the spectrum of $A$ into two
sufficiently separated parts is impossible. Then
\begin{equation}
\label{not-sep}
\eta_j \ge
\om
\frac
{\eta_{j+1}-2 x^2}
{\eta_{j+1}-2 x^2+\kappa}
\qquad \text{for} \enspace
j=t, \dots, m.
\end{equation}
We will prove that this leads to a contradiction.
Put $\dee=\om-\kappa/\si_0$.
Since $\si_0>K^2/\ga^2$, we have $0<\dee<\om$ and $\si_0=\frac \kappa {\om-\dee}$.

We prove that
\begin{equation}
\label{induct}
\eta_j\le
2x^2\,\frac{\si_0^{j-t+1}-1}{\si_0-1}
\end{equation}
for $j=t,\dots,m+1$ by induction in $j$. The induction base, $j=t$, follows from
our assumptions. Assume that \eqref{induct} holds for $j=j_0$, $t\le j_0\le m$.
By using that $|x| < \rho$ and \eqref{def-rho}, we get
$$
2x^2 < 2\rho^2\le \frac{\ga^2} 2 \;\mu\left(\frac \kappa \om\right)
=
\frac {\om \si_0-\kappa}
{2(\si_0+\si_0^2+\dots+\si_0^{m+1})}. 
$$
Therefore
$$
\eta_{j_0} \le 2x^2\, \frac {\si_0^{m+1}-1} {\si_0-1}
\le \frac 12 \, (\om - \frac \kappa {\si_0})=\frac \dee 2< \om .
$$
We also have
$\eta_{j_0+1}>2x^2$. Hence \eqref{not-sep} implies that
$$
\eta_{j_0+1} - 2x^2
\le
\kappa \frac
{ \eta_{j_0} }
{\om- \eta_{j_0} }
\le
\eta_{j_0}
\frac
{\kappa }
{\om- \de } = \eta_{j_0} \si_0.
$$
It follows that
$$
\eta_{j_0+1} \le \eta_{j_0} \si_0 +
2x^2 \le 2x^2 \Big[ \frac{\si_0^{j_0-t+1}-1}{\si_0-1} \si_0 + 1
\Big] = 2x^2 \,\frac{\si_0^{j_0-t+2}-1}{\si_0-1}.
$$
This gives the induction step. Hence \eqref{induct}
holds for all $j=t,\dots, m+1$.
In particular,
$$
\eta_{m+1}=\big|z-i\la_{\tau(m+1)}\big|^2
\le
2x^2 \, \frac{\si_0^{m+1}-1}{\si_0-1}.
$$
This gives a contradiction. Indeed, it follows that
$\la_{\tau(1)}, \dots, \la_{\tau(m+1)}$ are contained in the
interval $[y-\ell, y+\ell]$, where $\ell=(\eta_{m+1}-x^2)^{1/2}$.
Then
$$
\ell^2\le x^2
\Big(
2\, \frac{\si_0^{m+1}-1}{\si_0-1}-1
\Big)
<
\rho^2
\Big(
2\,\frac{\si_0^{m+1}-1}{\si_0-1}-1
\Big)
\le \frac {\De_m^2} 4
$$
(the last inequality is due to \eqref{def-rho}).
Hence $\ell^2< \frac {\De^2_m} 4$.
We get a contradiction to the definition of $\De_m$.
\end{proof}

\

Next, we take $s=s(z)$ as in the above Lemma and put
$$
P_\cl=P_{\lin\{v_{\tau(1)},\dots,v_{\tau(s)}\}},
\qquad
P_\far=P_{\lin\{v_{\tau(s+1)},\dots,v_{\tau(n)}\}}
$$
(recall that $v_j$ is an eigenvector of $A$ corresponding to $i\la_j$). Then
\begin{equation}
\Phi(z)=I-B^*(z-A)^{-2}B=\Psi+\Sigma,
\end{equation}
where
\begin{equation}
\label{PsiSigma}
\Psi  =I-B^*P_\far(z-A)^{-2}P_\far B, \qquad
\Sigma  =-B^*P_\cl(z-A)^{-2}P_\cl B.
\end{equation}

Put
$$
W=B^*P_\far,
$$
then
$$
\Psi=I-W(z-A)^{-2}W^*
=
I+\Psi_1,
$$
where
$$
\Psi_1\defin -W(z-A)^{-2}W^*.
$$
Define $\al>0$ from the equation
\begin{align}
\label{****}
\eta_\far=(2+\al)x^2.
\end{align}
 Then
\begin{equation}
\label{eq:sectors}
(y-\la_k)^2\ge (1+\alpha)x^2
\qquad
\text{for all}  \enspace i\la_k\in \si_\far(A).
\end{equation}
Hence
\begin{equation}
\label{9}
(yI+iA)^2-x^2 I\big|\Ran P_\far
\ge
\alpha x^2 I \big|\Ran P_\far.
\end{equation}

\begin{lemma}
\label{lm:Psi}
Suppose $\|B\|\le K$. Then
\begin{equation}
\label{RePsi}
\Re \Psi^{-1}\ge \be^{-1} I,
\end{equation}
where
$$
\be=1+\frac {K^2}{\alpha x^2}.
$$
\end{lemma}

\begin{proof}
A calculation gives
\begin{equation}
\label{10}
\Re \Psi_1 =- \frac 12 W
\big[
(z-A)^{-2} + (\bar z+A)^{-2}
\big]W^*
=W G W^*
\end{equation}
where
\begin{equation}
\label{11}
 G =\frac
{ (y+iA)^2-x^2I }
{ (z-A)^2 (\bar z+A)^2 }\Big| \Ran P_\far.
\end{equation}
We wish to prove that
\begin{equation}
\label{Psi-to-min1}
\Psi^{-1}+ \Psi^{*-1}\ge 2 \be^{-1} I.
\end{equation}
First let us check the inequality
\begin{equation}
\label{12}
\Psi^*\Psi \le \frac \be {2 } \big(\Psi+\Psi^*\big) =\be \Re \Psi.
\end{equation}
Inequality \eqref{12} is obtained as follows:
\begin{align*}
\Psi^*\Psi & =
\big(
I+\Psi_1^*
\big)
\big( I+\Psi_1 \big) \\
& = I + 2 \Re \Psi_1 + \Psi_1^*\Psi_1 \\
& = I + 2 W G W^* +
W(\bar z+A)^{-2} W^* W (z-A)^{-2} W^* \\
& \le
I + 2 W G W^* +
K^2 W(\bar z+A)^{-2} (z-A)^{-2} W^* \\
& \le
I + \big(2+ \frac { K^2 } {\al x^2}\big) W G W^* .
\end{align*}
The last inequality is due to \eqref{9} and \eqref{11}.
By \eqref{10}, this implies
\begin{align}
\label{estim-Psi}
\Psi^*\Psi & \le
I + (1+\be) \Re \Psi_1 
 = (1+\be) \Re \Psi -\be  I.
\end{align}
Rewrite this inequality as
$$
\big(\Psi^*-\frac {\be+1} 2 \,I\big)
\big(\Psi-\frac {\be+1} 2 \,I\big)
\le \big(\frac {\be - 1}2 \big)^2 I
$$
or, equivalently,
$\| \Psi-\frac {\be+1} 2 I \|\le \frac {\be - 1}2$. This gives
the inequality
$\Re \Psi \le \big(\frac {\be+1} 2 + \frac {\be - 1}2 \big) I =\be I$.
Then by  \eqref{estim-Psi},
$\Psi^*\Psi\le (\be+1)\Re \Psi - \be I\le
\be \Re \Psi$, and we get \eqref{12}.

We also get that
$\si(\Psi)\subset \big\{z: \big|z-\frac {\be + 1}2\big|\le \frac {\be - 1}2\big\}$, which implies that
$\Psi$ is invertible.
So \eqref{Psi-to-min1} follows immediately from
\eqref{12}.
\end{proof}

\begin{lemma}
\label{lm:Invert_sum} Suppose $1\le s\le m$, $\Psi$ is an $m\times
m$ matrix satisfying $\Re \Psi^{-1}\ge \be^{-1} I$, $V$ is an
$m\times s$ matrix such that
\begin{equation}
\label{eq:estim-V}
\|Vd\|
\ge
\ga \|d\|, \qquad
d\in \BC^s
\end{equation}
and $\La$ is an $s\times s$ invertible matrix. If $\be, \ga$ are
positive and
\begin{equation}
\label{15}
\be^{-1} \ga^2  > \|\La^{-1}\|,
\end{equation}
then the matrix $\Psi+V \La V^*$ is invertible.
\end{lemma}

\begin{proof}
Suppose it is not,
\begin{align}
\label{Psi.w}
\big(
\Psi+V \La V^*
\big) w
=0
\end{align}
for some $w\in \BC^m$, $w\ne 0$. Put $V^*w=c\in \BC^s$. One gets
\begin{align*}
w&=-\Psi^{-1} V \La c, \\
-V^* \Psi^{-1} V \La c & = c.
\end{align*}
Hence
$$
\langle c, \La c \rangle=
-
\langle
V^* \Psi^{-1} V \La c, \La c
\rangle.
$$
Then, on one hand, one has
$$
\big|
\langle c, \La c \rangle
\big|
=
\big|
\langle \La^{-1} \La c, \La c \rangle
\big|
\le
\| \La^{-1} \| \cdot  \|\La c\|^2,
$$
and on the other,
\begin{align*}
\big|
\langle c, \La c \rangle
\big|
=
\big|
\langle
 V^* \Psi^{-1} V \La c, \La c
\rangle
\big|
\ge \Re \langle
V^* \Psi^{-1} V \La c, \La c
\rangle  
\ge \be^{-1} \|V\La c\|^2
\ge
\be ^{-1} \ga^2  \|\La c\|^2.
\end{align*}
These inequalities and \eqref{15} imply that $c=V^* w=0$.
By \eqref{Psi.w}, it follows that $\Psi w=0$, so that $w=0$, which contradicts to
the choice of $w$.
\end{proof}

\begin{proof}[The end of the proof of Theorem \ref{decay-gener-m}]
As before, we assume that some $z=x+iy\in \BC_-$ with $-\rho< x<0$
has been fixed.
Lemma \ref{lm:suff_separ}
gives us an index $s$, $s\le m$, which defines a partition of $\si(A)$
into two sufficiently separated parts, $\si_\cl(A)$ and $\si_\far(A)$.
Define $\Psi$ and $\Sigma$ from \eqref{PsiSigma}.

By Lemma \ref{lm:Psi}, \eqref{RePsi} holds.
Put
$$
V=B^*\big|\Ran P_\cl, \qquad \La= (z-A)^{-2}|\Ran P_\cl
$$
and apply Lemma \ref{lm:Invert_sum} to these two matrices and $\Psi$.
Since $\eta_\cl<\eta_\far$, it follows that there is some $k$,
$1\le k\le n-m+1$, such that all the indices $\tau(1), \dots, \tau(s)$ are contained
in the set $\{k,k+1, \dots, k+m-1\}$
(this is true even if $A$ has multiple eigenvalues). Therefore, by hypothesis (ii) of the Theorem,
$V$ satisfies \eqref{eq:estim-V}.
By \eqref{****} and \eqref{eq:suff-sep}, one has
$$
\be^{-1} \ga^2=\ga^2 \frac {\al x^2}{\al x^2+K^2}=
\ga^2\, \frac {\eta_\far -2 x^2}{\eta_\far -2 x^2+K^2 }
>\eta_\cl=\|\La^{-1}\|.
$$
Hence \eqref{15} holds.
So $\Phi(z)=\Psi+V \La V^*=\Psi+\Sigma$ is invertible, and therefore $z\notin \siclloop$.
This proves the Theorem.
\end{proof}

\section{A brief account of our estimates of $\gadecay$}
\label{briefaccount}

Here, for the reader's convenience, we gather all the above estimates.

\underline{Theorem \ref{thm:main} for $m>1$}: $\gadecay >\ellest$; see
\eqref{def-ellest} and \eqref{eq:circles};

\medskip

\underline{Theorem \ref{thm:main}  for $m=1$}:  $\gadecay >\ellest^1$.

\medskip

\underline{Corollary \ref{corol-phiks}}:
$\gadecay < \Gamma_+=\min \{(1+\varphi_k)\|b_k\| : \; \varphi_k < 1\}$;

 \hskip2.3cm $ \gadecay > \Gamma_- = \min_k (1-\varphi_k)\|b_k\|$\;
if
$\varphi_k < 1$ for all $k$,

\hskip2.3cm where $\varphi_k = \frac{2\|B\|^2}{(2-\sqrt{2})^2\delta_k^2}$,
$b_k= B^* v_k$ and $Av_k = i\la_k v_k$, $\|v_k\|=1$.

\medskip

\underline{Theorem \ref{decay-gener-m}}: $\gadecay\ge\rho$ for $m\ge 2$,
where $\rho$ is defined in \eqref{def-rho}.

\medskip

We notice also that if these statements provide several lower or upper bounds for
$\gadecay$, then, obviously, one can take the best one of these.

\section{Numerical examples}
\label{sec:num-ex}

\subsection{An example with $4$ states and $2$ controls}
\label{4states2controls}

Take $n=4$, $m=2$ and consider the matrices
\[
 A = i\;\left[\begin{array}{cccc}
         -b & 0 & 0 & 0\\
          0 & -a & 0 & 0\\
          0 & 0 & a & 0\\
          0 & 0 & 0 & b
          \end{array}\right],
          \qquad
 B = \left[\begin{array}{cc}
          1 & 0\\
          0 & 1\\
          1 & 0\\
          0 & 1
             \end{array}\right],
\]
where $a,b$ are positive real numbers. Consider the LQR problem for $(A,B)$
with  $Q = I$, $R = I$.
Table \ref{table:m2} collects the values of $\|X\|$ and of $\gadecay$
for different values of $a,b$, obtained by numerical calculations.
The last four columns of this Table show
the values of  $\Gamma_-$, $\Gamma_+$, $\ellest$ and $\rho$, which are
the lower and upper bounds for $\gadecay$ guaranteed by our theorems (see the
previous section for a brief account).

Row 1 shows that the bounds $\Gamma_-$, $\Gamma_+$ for
$\gadecay $ are very precise in the case of large separation of the
spectrum of $A$. In rows 2 and 3, one can see that as the separation
diminishes (and some $\phi_k$'s approach to $1$), the bounds $\Gamma_-$,
$\Gamma_+$ become much more vague.

In row 4, there is some $k$ with $\phi_k < 1$, but we do not have $\phi_k < 1$
for all $k$. Hence, only the upper bound $\Gamma_+$ from
Corollary~\ref{corol-phiks} holds, and
$\Gamma_-$ is not defined.

In rows 5, 6 and 7, $\phi_k \ge 1$ for all $k$. Hence
Corollary~\ref{corol-phiks} provides no bounds at all, and we do not show the
values of $\Gamma_-$, $\Gamma_+$. In these rows one can see how the lower
estimate $\rho$ for $\gadecay$ from Theorem \ref{decay-gener-m} can give better
results than $\ellest$ from Theorem \ref{thm:main}, especially if some
eigenvalues of $A$ are close together in comparison with $\|B\|$.

Part (2) of Theorem \ref{thm:main} and Theorem \ref{decay-gener-m}
show that if the minimal singular value of $B$ is large in
comparison with the diameter of the spectrum of $A$, then the
closed-loop spectrum divides in two parts: $m$ eigenvalues are in
the band $\Re z\in [-\|B\|,-\frac{\sqrt{6}}{4}\sigma_m]$ and the
resting $n-m$ eigenvalues lie in the band $\Re z\in
\big(-\sqrt{3}\Delta,-\max(\ellest,\rho)\big)$. Within the values
of $a,b$ in the table, this result only applies to rows 6 and 7. For instance,
for row 7, Part (2) of Theorem \ref{thm:main} yields that two
closed-loop eigenvalues lie in the band
$\Re z \in [-1.4142, -0.8660]$ and two others in the band $\Re z\in (-0.3811,
-0.0198)$. Numerical simulation shows that two eigenvalues of
$A-BF$ satisfy $\Re \nu_{1,2}\approx -1.4024$ and two others
satisfy $\Re \nu_{3,4}\approx -0.1062$.

Simulation also shows that in many cases,
the relative error in the estimate
$\gadecay>\ellest^1$, which
Theorem \ref{thm:main} gives for $m=1$,
is less than in the corresponding estimate for $m>1$.
(On the other hand, the quality of the control increases with the increase of $m$).

\begin{table}
\begin{center}
\begin{tabular}{|l||l|l||l|l||l|l|l|l|}
\hline
 & $a$   & $b$                   &     $\|X\|$  & $\gadecay$  &  $\Gamma_-$ &
$\Gamma_+$ & $\ellest$ & $\rho$ \\
\hline
$1$ & $15$ &      $45$              &     $1.0171$   &  $0.9999$   &  $0.9870$
 &
$1.0130$   & $0.7040$  &  $0.0806$ \\
$2$ & $5$ &      $15$               &     $1.0537$   &  $0.9988$   &  $0.8834$
 & $1.1166$   & $0.6799$  &  $0.0806$  \\
$3$ & $1.8$ &     $5.4$            &     $1.1667$   & $0.9910$    &  $0.1006$
&
$1.8994$   & $0.5403$  &  $0.0806$ \\
$4$ & $1$ & $10$ & $1.1031$ & $0.9960$ & $-$ & $1.1439$ & $0.3536$ & $0.0806$ \\
$5$ & $4$ & $4.1$              &     $1.1456$   & $0.9928$    &  $-$         &
$-$      & $0.0018$  &  $0.0806$ \\
$6$ & $0.2$ & $0.22$           &     $5.4750$   & $0.2199$    &  $-$     &  $-$
&
$7.07\cdot 10^{-5}$ &   $0.0396$  \\
$7$ & $0.1$  &  $0.11$         &       $10.2326$  & $0.1062$    &  $-$     &
$-$ & $1.77\cdot 10^{-5}$ &  $0.0198$ \\
\hline
\end{tabular}
\end{center}
\caption{Numerical results and bounds for the LQR problems, $m=2$.}
\label{table:m2}
\end{table}

\subsection{A control problem for a mechanical problem}
\label{mech}

In many practical problems there is a large choice of possible physical or geometric configurations of the
controller, which might make it necessary to solve a large amount of LQR optimization problems, in order to
find a good one in some alternative sense. We will be speaking about the search of
an LQR optimal regulator, which is also good in the sense that it has the largest possible
$\gadecay$.

In this subsection, we propose an algorithm which allows one
to reduce drastically this search, by making use of our theoretical estimates.
We will illustrate this algorithm on a simple mechanical system
(a very similar example has been considered in \cite{HenchHeKuMehrm98} in the presence of damping).
The same algorithm, in fact,
can be applied to the following general class of problems:
to optimize $\gadecay$ among a large finite family of LQR problems $(A, B_j)$,
with $A$ skew-Hermitian.
In other words, the system matrix $A$ is supposed to be fixed,
but there are several possible choices for the control matrix $B$.

This is not the only application of our bounds. We believe that
in many cases the control designer can apply
our results to obtain some a priori information on the systems in study.

Consider a one-dimensional massless string.
Attached to the string are $N$ equal point masses of mass $M$, that are placed
along it at equal distances $h$. It is assumed that the unperturbed
string occupies the interval $[0, (N+1)h]$ of the $x$ axis in an
$xy$ plane; the string is supposed to move only in this plane.
The two endpoints of the string are fixed, and it has constant tension $\tau >
0$.

The problem is to stabilize the string using $m$ controls, where
$1\le m\le N$. Namely, we choose point masses with numbers $j_1, j_2, \dots,
j_m$, where
$1\le j_1 < j_2 < \dots < j_m\le N$, and apply a force $u_k$ to the $j_k$
point mass in the direction $y$. Every configuration $(j_1, j_2, \dots, j_m)$ of
controls leads to its own linear quadratic control problem and to a corresponding
stable closed-loop system, which is optimal in the linear quadratic sense.
However, the exponential decay rates of these closed-loop systems will depend
on the chosen configurations of the control. The problem we discuss here
is to find the configuration $(j_1, j_2, \dots, j_m)$
which leads to the best exponential decay rate.

In the experiment, we have chosen the parameters $\tau/h = 10$, $M = 50$ and
$N=30$. We tried the values $m=1, 2, 3, 4, 5, 8$.
One can observe that $\gadecay$ depends much on the choice of the configuration
$\mathcal{J}=(j_1, j_2, \dots, j_m)$ (these are the numbers of the masses to
which
the control forces are applied). For example, if $m=5$, then the best value
of $\gadecay$ equals to $8.87\cdot 10^{-4}$, which is attained, for instance, for
$\mathcal{J}=(2, 5, 11, 19, 27)$, while for
$\mathcal{J}=(1, 2, 3, 29, 30)$ one only gets
$\gadecay= 2.80\cdot 10^{-4}$, which is several times less.

There are $\binom N m$ configurations, and theoretically,
the problem can be solved by a ``brute force'' complete search among all of them.
However, even for moderate values of $N$ and $m$, solving
numerically  $\binom N m$ LQR problems will be very time-consuming.

If the position of the $j$-th point mass is $(jh, y_j)$, we obtain
(in the linear approximation) the following system of ODEs:

\[
 \begin{cases}
  My_j'' = \frac{\tau}{h}(y_{j+1} + y_{j-1} - 2y_j), & j = 1,\ldots,N, \quad
j\ne j_k, \\
  My_j'' = \frac{\tau}{h}(y_{j+1} + y_{j-1} - 2y_j)+u_k, & j = j_k, \quad
k=1,\dots, M, \\
  y_0 = y_{N+1} = 0,
 \end{cases}
\]
where $y_0, y_{N+1}$ have only been
introduced for convenience in the notation.

Put
\[
 A_0 = \begin{pmatrix}
        2 & -1 & 0 & 0 & \cdots & 0 & 0\\
       -1 & 2 & -1 & 0 & \cdots & 0 & 0\\
        0 & -1 & 2 & -1 & \cdots & 0 & 0 \\
        \cdots & \cdots & \cdots & \cdots & \cdots & \cdots & \cdots \\
    0 & 0 & 0 & 0 & \cdots & 2 & -1 \\
    0 & 0 & 0 & 0 & \cdots & -1 & 2 \\
       \end{pmatrix}_{N \times N},
\quad
A = \begin{pmatrix}0 & I \\ -\frac{\tau}{hM}A_0 & 0\end{pmatrix},
\quad
x = \begin{pmatrix}y_1 \\ \vdots \\ y_N \\ y_1' \\ \vdots \\ y_N'\end{pmatrix},
\]
\[
 B = \frac{1}{M}B_0, \qquad
 B_0 = \begin{pmatrix}
  0_{N\times m}\\
  \hline
  e_{j_1},\ldots,e_{j_m}\\
  \end{pmatrix}_{2N \times m},
\]
where $e_j \in \mathbb{C}^N$ is the $j$-th (column) vector of the canonical basis. Then
we obtain the control system
\[
x' = Ax + Bu.
\]

The energy of the system can be defined in terms of the following inner
product in $\mathbb{C}^{2N}$:
\[
 \langle f, g \rangle_E = \frac{1}{N+1}\left[
\frac{\tau}{hM}\Big(\sum_{k=1}^{N-1}(f_{k+1} - f_k)(\overline{g}_{k+1} -
\overline{g}_k) + f_1\overline{g}_1 + f_N\overline{g}_N\Big) +
\sum_{k=1}^N f_{N+k}\overline{g}_{N+k}\right].
\]
The energy is $E(x) = \tfrac{1}{2}\|x\|^2_E$. It is easy to
show that energy is conserved, so that $A$ is skew-Hermitian with respect to
this inner product.

Now we apply the Linear Quadratic Regulator using the cost functional
\[
 J^u(x_0) = \int_0^\infty \|x(t)\|^2_E + \|u(t)\|^2 dt
\]
in order to stabilize the system.

We can do a theoretical study of the system to obtain expressions to
compute our estimates. Notice that our string is a very particular case of a
nonhomogeneous string, whose spectral theory comes back to M.G. Krein,
see \cite[Section 8 of Chapter VI]{Gohb_Kr_Volterr}.
In our case, the eigenvalues of $A$ are
\[
i \lambda_k = -2 i\, \sqrt{\frac{\tau}{hM}}\sin\left(\frac{k\pi}{2(N+1)}\right),
\quad
-N \leq k \leq N, k \neq 0,
\]
and the corresponding orthonormal eigenvectors are $v_k$, where
\[
 v_k = \begin{pmatrix}\frac{1}{i\lambda_k}w_k\\ w_k\end{pmatrix},
\qquad w_k = \left(\sin\left(\frac{kl\pi}{N+1}\right)\right)_{1\leq l
\leq N} \;   .
\]
See the paper \cite{Micu2002} by Micu, where the same matrix $A$ appeared
in the context of a semidiscrete numerical scheme for 1D wave equation.
We also refer to \cite{BelishPo90}, \cite{MitrTarak2003} for a
related inverse problem.

The operator $\sqrt{N+1}\,B_0$ maps the canonical basis
of $\mathbb{C}^m$ onto an orthonormal system of $m$ vectors in
$\mathbb{C}^{2N}$
(we use the inner product $\langle\cdot,\cdot\rangle_E$ in
$\mathbb{C}^{2N}$ and the standard one in $\mathbb{C}^m$). Hence,
$\sqrt{N+1}\,B_0$ is an isometry and it follows that
\[
 \|B\| = \sigma_m(B) = \frac{1}{M\sqrt{N+1}}.
\]
Finally, the vectors $b_k = B^*v_k$ can be computed to obtain
\[
 \|b_k\|^2 = \frac{1}{M^2(N+1)^2}\sum_{l=1}^m \sin^2\left(\frac{j_l k
\pi}{N+1}\right), \quad -N \leq k \leq N, k\neq0.
\]

Using Corollary~\ref{corol-phiks}, we can give an upper bound for
$\gadecay$, assuming that some $\varphi_k < 1$. Theorem~\ref{thm:main} and
Corollary~\ref{corol-phiks} (if it applies) can be used to obtain a lower bound
for $\gadecay$. The following algorithm uses these
bounds to reduce the number of LQR problems being computed.
In the course of its execution,
the upper and the lower theoretical bounds for all configurations are taken into account, but
the LQ optimal regulator is actually computed for a fewer number of configurations.

The algorithm works as follows:

\begin{enumerate}
\item Calculate the eigenvalues $i\lambda_k$ and the corresponding eigenvectors
$v_k$ of $A$.

\item For each control configuration $\mathcal{J} = (j_1,\ldots,j_m)$, compute
the vectors $b_k$ and the quantities
$U_{\mathcal{J}}$ and $L_{\mathcal{J}}$, which are the upper and the lower
theoretical bounds for
$\gadecay$. Set $U_{\mathcal{J}} = +\infty$ if an upper bound is not available.

\item
Select the configuration $\mathcal{J}_0$ having the maximal
$L_{\mathcal{J}}$. Solve the LQR problem numerically for this configuration and
compute $\gadecay$.

\item
Now we proceed to a search, defined recursively as follows.
Let $\gamma$ be the best $\gadecay$ found so far. If
for all configurations $\mathcal{J}$ whose corresponding
$\gadecay$ has not been computed yet,
$U_{\mathcal{J}} < \gamma$, the search stops, and
this current value of $\gamma$ is taken for the optimal $\gadecay$.
If there are configurations $\mathcal{J}$
whose $\gadecay$ has not been computed that have
$U_{\mathcal{J}} \geq \gamma$, the algorithm selects the one
having the greatest $U_{\mathcal{J}}$.
For this configuration, it solves the LQR problem numerically,
computes its $\gadecay$ and updates $\gamma$ according to the rule $\gamma:=\max(\ga,\gadecay)$.
This is the best $\gadecay$ found so far.

\item
The algorithm stops after having exhausted
all possible configurations. It returns the last value of $\gamma$, which is equal to
the maximum of the values of $\gadecay$ over all possible configurations.

\end{enumerate}

Observe that this algorithm also allows one to compute all the configurations having the
optimal $\gadecay$.

\begin{table}
\begin{center}
\begin{tabular}{|l||l|l|l|l|}
\hline
$m$ & $\gadecay$ & Time (s) & LQRs computed  & \% computed \\
\hline
$1$ & \vphantom{\huge l }$6.53\cdot10^{-5}$ & $2$ & $30$ & $100$ \\
$2$ & $3.39\cdot10^{-4}$ & $4$ & $76$ & $17.5$ \\
$3$ & $5.82\cdot10^{-4}$ & $28$ & $441$ & $10.9$ \\
$4$ & $7.70\cdot10^{-4}$ & $123$ & $481$ & $1.76$ \\
$5$ & $8.87\cdot10^{-4}$ & $738$ & $5505$ & $3.86$ \\
$8$ & $1.2\cdot10^{-3}$ & $56675$ & $198369$ & $3.39$ \\
\hline
\end{tabular}
\end{center}
\caption{Results for $N = 30$, $\tau/h = 10$, $M = 50$.}
\label{table:simul2}
\end{table}

The results of the execution of the algorithm are shown on
Table~\ref{table:simul2}.
The computations were done on a
modern desktop computer.
Recall that we have chosen the total number of masses $N=30$.
The table shows that the decay rate $\gadecay$ improves when $m$ increases. The
fourth column collects the number of
LQRs the algorithm had to solve, and the fifth column shows the ratio between
the total of $\binom N m$ possible configurations
and the number of configurations that were actually processed.
One can see that in many cases, our algorithm reduces drastically the amount of computations.

The values $\tau/h = 10$ and $M = 50$ have been chosen for these
computations because they provide a moderate separation of the spectrum of $A$
with respect to $\|B\|$. If we fix $M = 50$ and increase $\tau/h$ (say $\tau/h
= 1000$), then the number of computations is further reduced, since the
separation of the spectrum of $A$ increases and we obtain tighter theoretical
bounds. On the other hand, if one sets $\tau/h$ to a small enough value while
maintaining $M$ fixed, our algorithm will not provide much save in the
computations.

\section{Some open questions}
\label{sec:open-que}

\textbf{Question 1. }
Assume that $m<n$, $R = I$, $Q = I$ and that a skew-Hermitian matrix $A$ is
fixed.
Does it follow that there is a constant $C=C(A)$ such that
$\gadecay \le C$, independently of $B$?
As we already mentioned in Corollary~\ref{gamma-bound}, it is true if $m=1$, with
$C(A)=2\sqrt{2}\, \De$. More generally, part (2) of
Theorem \ref{thm:main} shows that it is also true if, for instance,
$\si_m(B)\ge \frac 12 \|B\|$, or even if we assume that
$\si_m(B)\ge f(\|B\|)$, where $f$ is any function
on $[0, +\infty)$ such that $\lim_{x\to\infty} f(x)=+\infty$.
We conjecture that it is true in general.

\

\textbf{Question 2.} We can pose a somewhat related question concerning
the general pole placement problem for a general complex matrix $A$. Suppose that $m<n$, and let
$\gadecay$ denote the decay rate of the matrix of a stable closed loop system
$A_{\text{cl.loop}} = A - BF$, which is
obtained by (an arbitrary) state space control $u(t)=-Fx(t)$.
Can one assert that the cost matrix
$X_0=\int_0^\infty \exp(A_{\text{cl.loop}}^*t)\exp(A_{\text{cl.loop}}t)\, dt$
is large every time when $\gadecay$ is large?
We conjecture that it is so. Then, it would be interesting to find
an explicit function $G(x)$ (which may depend only on $n, m, A$),
that goes to infinity as $x\to+\infty$
and satisfies $\|X_0\|\ge G(\gadecay)$ for all $B, F$ such that
$A_{\text{cl.loop}}$ is stable. A weaker version of this question is whether
there is such function $G$ that may depend on both $A$ and $B$.

\

\textbf{Question 3.} Corollary~\ref{gamma-bound} can be used to obtain
an upper bound for $\gadecay$. However, if $\varphi_k < 1$ for some $k$, then
Corollary~\ref{corol-phiks} gives a much tighter bound. Can one give a tighter
upper bound even when $\varphi_k \geq 1$ for every $k$?

\section{Conclusions}
\label{conclusions}

\begin{itemize}

\item The bounds $\ellest$, $\ellest^1$ given in Theorem~\ref{thm:main} can be applied only if
all the eigenvalues of $A$ are different.

\item The lower bound $\rho$ given in Theorem~\ref{decay-gener-m} is the one which
can be used in a more general setting (namely $\Delta_m > 0$, which allows
some eigenvalues of $A$ to coincide).
There are cases when it is the best bound available.
It happens, in particular, if some
eigenvalues of $A$ are close together (compared with $\|B\|$).

 \item The two-sided bound given in Corollary~\ref{corol-phiks} holds only when
$\varphi_k < 1$ for all $k$, i.e., when the spectrum of $A$ is separated enough.

\item If all $\varphi_k$ are small, this two-sided bound is very tight and one
can take $d_0 =
\min \|b_k\|$ as a good approximation for $\gadecay$.

 \item When all $\varphi_k$ are small, one can also use
Theorem~\ref{zeros-disks} to locate with precision all the eigenvalues of the
closed-loop system.

\item
Corollary~\ref{gamma-bound} shows that if
$m < n$ and the diameter $\Delta$ of the spectrum of $A$ is much smaller than all
singular values of  $B$, then $\gadecay$ is less than $\sqrt{3}\,\Delta$.

\item One can observe that, as a rule, if the separation of the eigenvalues of
$A$ increases or the number of controls $m$ increases, then $\gadecay$
grows.

\item
If one has to find an optimal $\gadecay$ among a large finite family of LQR control problems,
our estimates permit one to design an algorithm to reduce the search (in some situations,
drastically; see Subsection \ref{mech}).

\item
By now, we only have estimates of $\gadecay$ for the case of a skew-Hermitian
matrix $A$.
It would be very desirable to
give good estimates of $\gadecay$ and $\|X\|$ for non-skew Hermitian matrices, or at least
for the case of matrices $A$ such that $\Re A\ge 0$. Another interesting subclass
are normal matrices $A$, for which some modifications of our methods
could apply. This can also be interesting for the stabilization method
we mentioned in Remark~(\ref{remark-shifted}) after
Theorem~\ref{thm:main-detailed}.

\end{itemize}

\section{Acknowledgements}


The first author has been supported by the JAE-Intro grant of the CSIC (Spanish
National Research Council) and the ICMAT-Intro grant of the Institute for
Mathematical Sciences, Spain.

The second author has been supported by the Projects MTM2008-06621-C02-01,
and MTM2011-28149-C02-1, DGI-FEDER, of
the Ministry of Science and Innovation of Spain,
and by ICMAT Severo Ochoa project SEV-2011-0087
(Spain).


\end{document}